\documentclass{amsart} 

\usepackage[french]{babel}
\usepackage[latin1]{inputenc}

\usepackage{amsmath}
\usepackage{amsthm} 
\usepackage{amssymb}
\usepackage{tikz}

\usepackage{imakeidx} 
\usepackage{appendix}
\usepackage{tikz-cd}
\usepackage{verbatim}
\usepackage{enumitem}
\usepackage{multicol}

\usepackage[backref=page]{hyperref}
\usepackage{cleveref}

 \hypersetup{colorlinks=true,linkcolor=blue,anchorcolor=blue,citecolor=blue}

\usepackage{adjustbox}

\usepackage{mathtools}

\newtheorem{thm}{Th\'eor\`eme}[section]
\newtheorem{prop}[thm]{Proposition}
\newtheorem{lemma}[thm]{Lemme}

\theoremstyle{definition}

\theoremstyle{remark}

\newtheorem{rmk}[thm]{Remarque}
\numberwithin{equation}{section}

\newcommand{\Q}{\mathbb Q}

\newcommand{\C}{\mathbb C}
\newcommand{\Z}{\mathbb Z}

\renewcommand{\P}{\mathbb P}
\newcommand{\Spec}{\operatorname{Spec}}

\renewcommand{\O}{\mathcal O}

\renewcommand{\phi}{\varphi}

 \def\L{{\mathcal L}}
 
  \def\k{\overline{k}}
  \def\X{\overline{X}}
\def\Pic{{\rm Pic}}

\title[Z\'ero-cycles]{Z\'ero-cycles sur les surfaces de del Pezzo (Variations sur un th\`eme de Daniel Coray)}

\author{Jean-Louis Colliot-Th\'el\`ene}

\address{Universit\'e Paris-Saclay, CNRS, Laboratoire de math\'ematiques d'Orsay, 91405, Orsay, France.}
\email{jlct@math.u-psud.fr}

\date{27 septembre 2020}

\begin{document}

	\maketitle
	\hypersetup{backref=true}

	\def\P{{\bf P}}

\bigskip

{\it R\'esum\'e}
Soit   $X$ une surface projective, lisse, g\'eom\'etriquement rationnelle sur un corps 
de caract\'eristique z\'ero. On lui associe deux entiers $N(X)$ et $M(X)$,  fonctions 
 simples du carr\'e de la classe canonique. On \'etablit  les propri\'et\'es suivantes.

(a) Si le pgcd des degr\'es des points ferm\'es sur $X$ est $1$, alors il existe des points
ferm\'es dont les degr\'es  sont  au plus \'egaux \`a  $N(X)$ et sont premiers entre eux dans leur ensemble. 

(b) Si $X$ poss\`ede un point rationnel, alors tout z\'ero-cycle sur $X$
de degr\'e au moins \'egal \`a $M(X)$  est rationnellement \'equivalent \`a un z\'ero-cycle effectif,
 et les points ferm\'es
de degr\'e au plus \'egal \`a  $M(X)$ engendrent le groupe de Chow des z\'ero-cycles de $X$.
 
Le r\'esultat (a)  g\'en\'eralise un th\'eor\`eme de Daniel Coray  sur les surfaces cubiques  (1974).
  Une combinaison de th\'eor\`emes de Bertini
et d'utilisation de corps fertiles rend ici ses arguments plus flexibles.
On \'etablit ensuite les r\'esultats par consid\'eration des diff\'erents types birationnels
de surfaces g\'eom\'etriquement rationnelles : surfaces de del Pezzo et surfaces
fibr\'ees en coniques (ces derni\`eres d\'ej\`a \'etudi\'ees avec D. Coray en 1979).

Un dernier paragraphe discute l'existence de points ferm\'es de degr\'e 3  non align\'es
sur une surface cubique sans point rationnel. On la relie \`a la question de la
densit\'e des points rationnels sur une surface de del Pezzo de degr\'e 1.

\medskip
  
{\it Summary}
Let $X$ be a smooth, projective, geometrically rational surface over a field
of characteristic zero. To any such surface one associates two integers $N(X)$
and $M(X)$ which are simple functions of the square of the canonical class.
We prove:

(a) If the gcd of the degrees of closed points on $X$ is $1$, then there exist
closed points on $X$ the degrees of which are coprime to one another as a whole
and are   less than or equal to $N(X)$.

(b) If $X$ has a rational point, then any zero-cycle on $X$ of degree at least 
equal to $M(X)$ is rationally equivalent to an effective cycle. Effective zero-cycles
of degree less than or equal to $M(X)$ generate the Chow group of $X$.

Result (a) extends a theorem on   cubic surfaces 
obtained by Daniel Coray in his thesis (1974). Combining Bertini theorems
and large  fields, we introduce some flexibility in his method.
The  results (a) and (b) then follow from
a case by case analysis of the various birational equivalence classes of geometrically rational surfaces :
 del Pezzo surfaces and conic bundle surfaces  
 (the latter type had been handled with D. Coray in 1979).
 
In a last section,  for smooth cubic surfaces without a rational point,
we relate the question whether there exists a degree 3 point 
which is not on a line to the question whether rational points are dense
on a del Pezzo surface of degree 1.

\section{Introduction}

 Soient $k$ un corps et $X$ une $k$-vari\'et\'e alg\'ebrique, par quoi l'on entend un $k$-sch\'ema
s\'epar\'e de type fini sur $k$. Donnons quelques rappels sur les z\'ero-cycles
et l'\'equivalence rationnelle \cite[Chap. I]{F}.
Un z\'ero-cycle sur $X$ est une combinaison lin\'eaire \`a coefficients entiers de points ferm\'es.
\`A tout tel z\'ero-cycle $z=\sum_{P}n_{P}P$, avec $P$ point ferm\'e et $n_{P} \in \Z$ nul sauf pour
un nombre fini de points ferm\'es $P$, on associe son degr\'e 
$${\rm deg}_{k}(z): = \sum_{P} n_{P} [k(P):k] \in \Z,$$
o\`u $k(P)$ est  le corps r\'esiduel d'un point ferm\'e $P$, et $[k(P):k]$ est le degr\'e de cette
extension finie de corps.  Le groupe ab\'elien libre $Z_{0}(X)$ des z\'ero-cycles 
contient le sous-groupe
des z\'ero-cycles rationnellement \'equivalents \`a z\'ero. Celui-ci est  par d\'efinition
engendr\'e par les z\'ero-cycles de la forme $p_{*}({\rm div}_{C}(f))$, o\`u $C$ est une 
courbe sur $k$, normale, int\`egre, de corps des fonctions rationnelles $k(C)$, o\`u $p: C \to X$
est un $k$-morphisme propre,
o\`u
$f \in k(C)^*$ est une fonction rationnelle non nulle sur $C$, o\`u ${\rm div}_{C}(f)) \in Z_{0}(C)$
est son diviseur sur $C$, qui est un z\'ero-cycle sur $C$, et o\`u $p_{*} : Z_{0}(C) \to Z_{0}(X)$
est l'image directe par morphisme propre.   Le quotient de $Z_{0}(X)$ par le sous-groupe des
z\'ero-cycles rationnellement \'equivalents \`a z\'ero est appel\'e groupe de Chow des z\'ero-cycles sur $X$
et est not\'e $CH_{0}(X)$. Lorsque $X$ est propre sur $k$, l'application ${\rm deg}_{k}: Z_{0}(X) \to \Z$
passe au quotient par l'\'equivalence rationnelle (puisque le degr\'e du diviseur d'une fonction rationnelle sur une courbe
propre, normale, int\`egre, est nul). On a donc dans ce cas une application induite ${\rm deg}_{k} : CH_{0}(X) \to \Z$.
On note alors $A_{0}(X)$ le noyau de cette application.

\medskip

Un z\'ero-cycle $\sum_{P}n_{P}P$ est dit effectif si l'on a $n_{P} \geq 0$
pour tout $P$. Il y a identification entre l'ensemble $X(k)$ des points $k$-rationnels de $X$
et l'ensemble des  z\'ero-cycles effectifs de degr\'e 1 de $X$.

\medskip

Si $X$ est une courbe $C$  projective, lisse, g\'eom\'etriquement connexe de genre $g $ sur le corps $k$,
l'in\'egalit\'e de Riemann   pour la courbe $C$ 
montre que  l'on a les propri\'et\'es suivantes :

(i)  Tout z\'ero-cycle de $C$ de degr\'e au moins \'egal \`a $g$
est rationnellement \'equivalent \`a un z\'ero-cycle effectif. 
 
(ii) Si $C$ poss\`ede un z\'ero-cycle de degr\'e 1, alors
la courbe $C$ poss\`ede un z\'ero-cycle effectif
de degr\'e $g$ et un z\'ero-cycle effectif de degr\'e $g+1$.

(iii) Si $g\geq 1$ et $C$ poss\`ede un point $k$-rationnel $P_{0}$, le groupe de
 Chow $CH_{0}(X)$ est engendr\'e par les points ferm\'es $P$ de degr\'e 
${\rm deg}_{k}(P) \leq g$. En effet, pour tout z\'ero-cycle $z$, le z\'ero-cycle
$z +(g- {\rm deg}_{k}(z)) P_{0}$ est de degr\'e $g$.

(iv) Si $C$ est de genre $0$ ou $1$ et poss\`ede un z\'ero-cycle de degr\'e 1,
alors $C(k) \neq \emptyset$.

\medskip

On peut se demander dans quelle mesure ces belles propri\'et\'es des z\'ero-cycles sur les courbes
s'\'etendent aux z\'ero-cycles sur les vari\'et\'es de dimension quelconque.

  \medskip
  
Pour  les vari\'et\'es projectives, lisses, connexes sur  un corps alg\'ebriquement clos de degr\'e de transcendance infini sur le corps premier, en particulier $k=\C$ le corps des complexes, une propri\'et\'e comme (i) impose de s\'ev\`eres restrictions \`a la g\'eom\'etrie de la vari\'et\'e consid\'er\'ee.
Ceci a fait l'objet de  travaux  bien connus de   Mumford et  de Bloch. De mani\`ere  g\'en\'erale, 
pour $X/\C$ une telle vari\'et\'e, s'il existe un entier $d(X)>0 $ tel que tout z\'ero-cycle sur $X$ de degr\'e au moins $d(X)$
est rationnellement \'equivalent \`a un z\'ero-cycle effectif, alors pour tout $i \geq 2$ les groupes de cohomologie coh\'erente $H^{i}(X,\O_{X})$ sont nuls (Bloch et Srinivas \cite{BS}). 
  Pour $X/\C$ une surface projective et lisse, c'est une conjecture de  Bloch qu'inversement la condition $H^2(X,\O_{X})=0$
  implique l'existence d'un entier $d=d(X)$  comme ci-dessus. 
  
 Sur un corps $k$ quelconque,  il est alors naturel de se limiter  aux $k$-vari\'et\'es $X$ projectives, lisses, g\'eom\'etriquement connexes    qui v\'erifient  : il existe un entier $d_{geom}(X)$ tel que, sur {\it tout} corps alg\'ebriquement clos $\Omega$ contenant $k$,  tout z\'ero-cycle de degr\'e au moins $d_{geom}(X)$ sur $X \times_{k}\Omega$ est rationnellement \'equivalent \`a un z\'ero-cycle effectif.
 
  Parmi celles-ci, on trouve les vari\'et\'es g\'eom\'etriquement rationnellement connexes,
  au sens de Koll\'ar, Miyaoka et Mori.  Pour ces vari\'et\'es, on a  $A_{0}(X_{\Omega})=0$, et donc  
  $d_{geom}(X)=1$ convient.

Une classe bien \'etudi\'ee de telles vari\'et\'es est celle des mod\`eles projectifs et lisses
d'espaces  homog\`enes de groupes alg\'ebriques lin\'eaires connexes.  En ce qui concerne l'analogue de
la question  (iv)  ci-dessus, \`a savoir si l'existence d'un z\'ero-cycle de degr\'e~1 implique
celle d'un point rationnel, pour les compactifications lisses d'espaces principaux homog\`enes
de groupes alg\'ebriques lin\'eaires connexes, ceci a \'et\'e \'etabli dans de nombreux cas (Serre, Sansuc, Bayer--Lenstra),  mais des contre-exemples pour les espaces homog\`enes g\'en\'eraux ont \'et\'e donn\'es par Florence et  par Parimala.
L'\'enonc\'e historique concerne les quadriques : une quadrique  qui poss\`ede un point dans une extension
de degr\'e impair poss\`ede un point rationnel. Il fut conjectur\'e par  Witt (1937), d\'emontr\'e par
  Artin (non publi\'e, 1937) et  par Springer \cite{Spr}.

 En dimension 2, la classe des vari\'et\'es (s\'eparablement) rationnellement connexes co\"{\i}ncide avec celle des $k$-surfaces g\'eom\'etriquement rationnelles,
pour lesquelles on dispose de la classification $k$-birationnelle de Enriques, Manin, Iskovskikh, Mori :
tout telle surface est $k$-birationnelle soit \`a une surface fibr\'ee en coniques sur une conique, soit
\`a une surface de del Pezzo.
Dans cet article, nous \'etudions syst\'ema\-ti\-quement les \'enonc\'es de type (i), (ii), (iii)  pour les surfaces de del Pezzo.
En combinaison avec l'\'etude des z\'ero-cycles sur les surfaces fibr\'ees en coniques faite avec   Coray \cite{CTC},
ceci \'etablit les th\'eor\`emes suivants, analogues des \'enonc\'es (i), (ii), (iii) ci-dessus pour les courbes.\footnote{
L'analogue de l'\'enonc\'e  (iv) est connu, et rappel\'e dans la d\'emonstration du th\'eor\`eme  \ref{pluspetitpremier} :  pour une $k$-surface $X$ projective, lisse, g\'eom\'etriquement rationnelle avec $(K_{X}.K_{X}) \geq 4$, si $X$ poss\`ede un z\'ero-cycle de degr\'e 1, alors $X$ poss\`ede un point rationnel.}
 
\medskip

{\bf Th\'eor\`eme A} (Th\'eor\`eme \ref{pluspetitpremier})
 {\it Soit  $X$ une $k$-surface projective, lisse, g\'eom\'e\-tri\-quement rationnelle,
sur un corps $k$ de caract\'eristique z\'ero.  Soit $K_{X}$
 la classe canonique de $X$. Soit
  $$N(X) = {\rm max}(10, \lfloor{4- (K_{X}.K_{X})/2) }\rfloor    ).$$
Si $X$ poss\`ede un z\'ero-cycle de degr\'e 1, alors $X$
poss\`ede des points ferm\'es dont les  degr\'es sont  inf\'erieurs ou \'egaux \`a $N(X)$
et sont premiers entre eux  dans leur ensemble.}

\medskip

{\bf Th\'eor\`eme B} (Th\'eor\`eme  \ref{touseffectifs})
{\it Soit  $X$ une $k$-surface projective, lisse, g\'eom\'e\-tri\-quement rationnelle,
sur un corps $k$ de caract\'eristique z\'ero. Soit $K_{X}$
 la classe canonique de $X$.
Supposons que $X$  poss\`ede un point $k$-rationnel.
Soit $$M(X)= {\rm max}(904, \lfloor{3- (K_{X}.K_{X})/2)}\rfloor ).$$
Tout
 z\'ero-cycle de degr\'e au moins $M(X)$ est rationnellement
 \'equivalent \`a un z\'ero-cycle effectif. En particulier, le groupe de Chow des z\'ero-cycles est engendr\'e 
 par les points ferm\'es de degr\'e au plus  $M(X)$.
}
 
 \medskip
 
 Ceci pose deux  questions :
 
 (1) Peut-on \'etablir ces \'enonc\'es, avec des entiers $N(X)$ et $M(X)$ ne d\'ependant que de la g\'eom\'etrie de $X$
 sur une cl\^{o}ture alg\'ebrique du corps de base $k$,
 sans utiliser la classification $k$-birationnelle des surfaces
 g\'eom\'etri\-que\-ment rationnelles  et une analyse cas par cas ?
 
 (2) A-t-on des analogues de ces    \'enonc\'es  pour les vari\'et\'es rationnellement connexes
 de dimension sup\'erieure ?

\medskip

Le point de d\'epart de cet article est la th\`ese  de Daniel Coray (Cambridge, UK, 1974) \cite{C1}.
Daniel Coray  y montra que
 si une surface cubique lisse $X$ d\'efinie sur un corps $k$ parfait
 poss\`ede un point rationnel dans une extension finie de corps $K/k$
 de degr\'e premier \`a 3, alors elle poss\`ede un point rationnel
 dans une extension de corps $K/k$ de degr\'e soit 1, soit 4, soit 10.
  Voici le principe de sa d\'emonstration. On consid\`ere un point ferm\'e
$P$ de degr\'e premier \`a 3 aussi petit que possible, on fait passer par ce point et par
un point de degr\'e 3 une surface de $\P^3_{k}$  de degr\'e aussi petit
que possible, pour que le genre arithm\'etique $p_{a}$ de la courbe intersection $\Gamma$ 
soit aussi petit que posssible.
Si cette courbe est lisse et g\'eom\'etriquement connexe
de genre $g=p_{a}$,
on applique le th\'eor\`eme de Riemann-Roch sur la courbe $\Gamma$ \`a un z\'ero-cycle de degr\'e 
au moins \'egal \`a $g$, premier \`a $3$, et de degr\'e aussi petit que possible.
Dans les bons cas, on \'etablit ainsi l'existence sur $\Gamma$ 
et donc sur $X$ d'un z\'ero-cycle effectif, et donc d'un point ferm\'e, de degr\'e premier \`a 3
 plus petit que celui que l'on avait au d\'ebut, et on recommence le proc\'ed\'e.
 Le processus a ses limites : on n'arrive pas \`a r\'esoudre les cas $4$ et $10$,
 dont la possibilit\'e \`a ce jour n'est pas exclue.
 
 La m\'ethode fut ensuite appliqu\'ee par Coray    
  \cite{C2}  aux surfaces de del Pezzo de degr\'e  4, 
   et une variante fut appliqu\'ee par Coray et moi \cite{CTC} aux surfaces fibr\'ees en coniques
  sur la droite projective.
 
Une  difficult\'e technique dans ces articles est que les courbes obtenues
dans un syst\`eme lin\'eaire donn\'e  ne sont pas a priori lisses ni m\^eme
g\'eom\'etri\-que\-ment irr\'eductibles : on doit donc consid\'erer et discuter
 les  d\'eg\'en\'erescences possibles.
 
 \medskip

Voici maintenant le contenu d\'etaill\'e de l'article.

  Au \S \ref{bertinifertile},  on donne un argument nouveau et g\'en\'eral, combinant th\'eor\`eme de Bertini, 
  d\'efor\-ma\-tion  et sp\'ecialisation, qui dans ce type d'argument  permet
     de ne consid\'erer que le cas des courbes lisses.
La souplesse obtenue nous permet de d\'evelopper l'argument de Coray dans plusieurs
directions.\footnote{B. Poonen m'a   tr\`es r\'ecemment fait remarquer que l'utilisation des corps fertiles
et en particulier des corps de s\'eries formelles $k((t))$ pourrait souvent
\^etre remplac\'ee par une utilisation du lemme de Lang-Nishimura (\cite[Lemme 3.1.1]{CTCS}, \cite[Thm. 3.6.11]{Poo}) comme  c'est fait dans
 \cite[Lemma 9.4.8]{Poo}. Les deux types d'arguments sont en fait tr\`es proches.}

Au \S \ref{cubique1} on reprend l'argument de Coray  \cite{C1} pour les surfaces cubiques.
 
Au \S \ref{dp3pointrat}, on montre que si une surface cubique lisse poss\`ede 
un point rationnel, alors le groupe de Chow des z\'ero-cycles est engendr\'e par les 
points rationnels et les points ferm\'es de degr\'e 3, et que tout z\'ero-cycle de degr\'e
au moins 10 est rationnellement \'equivalent \`a un z\'ero-cycle effectif.
  
Au \S \ref{137},     on \'etablit pour les surfaces de del Pezzo de degr\'e 2 
l'analogue   du r\'esultat de Coray pour les surfaces de del Pezzo de degr\'e 2.
 On montre ici que s'il y a un point dans une extension finie de corps de degr\'e impair,
 alors il y a un point dans une extension de degr\'e 1, 3 ou 7.  Le degr\'e minimal impair 3 ne peut \^etre exclu 
 (\cite{KM}, voir la remarque \ref{contrexKM}).
 
 Au \S \ref{dp2pointrat}, on montre que si une surface de del Pezzo de degr\'e 2  poss\`ede 
un point rationnel, alors tout z\'ero-cycle de degr\'e 0 est rationnellement \'equivalent \`a la diff\'erence de
deux z\'ero-cycles effectifs de degr\'e 6, et que tout z\'ero-cycle de degr\'e au moins 43 est 
rationnellement \'equivalent \`a un z\'ero-cycle effectif.

 Au \S \ref{dp1pointrat}, on montre que sur une surface de del Pezzo de degr\'e 1   tout z\'ero-cycle de degr\'e 0 est rationnellement \'equivalent \`a la diff\'erence de
deux z\'ero-cycles  effectifs de degr\'e 21, et que tout z\'ero-cycle de degr\'e au moins 904 est 
rationnellement \'equivalent \`a un z\'ero-cycle effectif.
  
Au \S \ref{total}, 
on note que ces divers r\'esultats, combin\'es avec   \cite{CTC},
 ach\`event la d\'emons\-tration  des th\'eor\`emes A et B mentionn\'es ci-dessus.
 
  \medskip

Au \S \ref{cubique2}, logiquement ind\'ependant du reste de l'article,
on revient aux surfaces cubiques.
On s'int\'eresse \`a une question soulev\'ee par Qixiao Ma \cite{QM} :
sur une surface cubique lisse sans point rationnel, existe-t-il un point
ferm\'e de degr\'e 3 non d\'ecoup\'e par une droite ? On relie  ce probl\`eme
\`a la question (ouverte) de la densit\'e des points rationnels sur les  surfaces de del Pezzo
de degr\'e 1.

Pour ne pas alourdir ce texte, on se limite aux corps de caract\'eristique nulle.
On laisse au lecteur le soin de voir ce qui subsiste sur un corps quelconque.
Des r\'esultats dans cette direction sont obtenus dans \cite{C1} et \cite{QM}.
 
On utilise librement dans cet article la th\'eorie de l'intersection sur les surfaces projectives lisses
\cite{Se, Mu}, 
 la th\'eorie des  surfaces cubiques et plus
g\'en\'eralement des surfaces de del Pezzo comme on peut la trouver dans   les livres 
  \cite{Ma} et  \cite{Kol}, dans \cite{Ma66} et \cite{I}, et dans le  rapport \cite{VA}.
  On utilise aussi des r\'esultats 
 sur  les vari\'et\'es rationnellement connexes, \'etablis par   les techniques de d\'eformation
 de Koll\'ar, Miyaoka et Mori  \cite{Kol, KoAM}.
 
 \medskip
 
 Daniel Coray, qui fut professeur \`a l'Universit\'e de Gen\`eve,  et fut aussi
 directeur de publication de la revue l'Enseignement Math\'ematique,
 nous a quitt\'es en 2015. 
 C'\'etait un esprit fin et original. 
 Les d\'emons\-trations de l'article \cite{CTCS}, o\`u un substitut du principe de Hasse
 fut \'etabli pour la premi\`ere fois pour une classe de vari\'et\'es ne se ramenant pas
 par transformations birationnelles
\`a des espaces homog\`enes de groupes alg\'ebriques lin\'eaires, en gardent la trace.
 \`A ce sujet on pourra aussi consulter ses {\it Notes de G\'eom\'etrie et
 d'Arithm\'etique}, r\'ecemment traduites \cite{C3}.  
  Je suis heureux de pouvoir d\'edier cet article \`a sa m\'emoire.

\section{Bertini et corps fertiles}\label{bertinifertile}

\subsection{Corps fertiles et $R$-densit\'e}

 Un corps $F$ est dit fertile (terminologie due \`a Moret-Bailly; en anglais on dit \og large field \fg)  s'il satisfait la propri\'et\'e suivante :
si une $F$-vari\'et\'e  $X$ int\`egre poss\`ede un $F$-point lisse, alors
l'ensemble $X(F)$ de ses points $F$-rationnels est dense dans $X$
pour la topologie de Zariski.  On consultera \cite{P} pour un rapport r\'ecent sur le sujet.
Une extension finie d'un corps fertile est fertile.
Pour tout corps~$k$, le corps  $F=k((t))$ des s\'eries formelles sur $k$ est fertile. 

 Soient $k$ un corps et $X$ une $k$-vari\'et\'e propre.
  La relation de $R$-\'equivalence sur $X(k)$ est par d\'efinition
  engendr\'ee par la relation \'el\'ementaire suivante : deux $k$-points
  $A,B \in X(k)$ sont \'el\'ementairement li\'es s'il existe un $k$-morphisme
  $f: \P^1_{k} \to X$ tel que $A$ et $B$ sont dans $f(\P^1(k))$. 
  Si  deux  $k$-points  $P$ et $Q$ sont $R$-\'equivalents, alors le z\'ero-cycle $P-Q$
  est rationnellement \'equivalent \`a z\'ero sur $X$.
 
 Soient $k$ un corps de caract\'eristique z\'ero et  $X$
 une $k$-vari\'et\'e  projective, lisse, g\'eom\'etriquement connexe. 
 On dira dans cet article que la $k$-vari\'et\'e $X$ satisfait la {\it propri\'et\'e de densit\'e} si,
 pour toute extension finie de corps $L/k$ telle que $X(L)\neq \emptyset$,
 l'ensemble $X(L)$ est  dense dans $X_{L}$  pour la topologie de Zariski.
 Sur un corps $k$ fertile, toute  $k$-vari\'et\'e lisse g\'eom\'etriquement connexe 
 $X$ satisfait la propri\'et\'e de densit\'e.
 
 Une hypersurface cubique lisse dans $\P^n_{k}$ pour $n \geq 3$ est 
$k$-unirationnelle d\`es qu'elle a un point rationnel  (voir \cite{Ko02}).
Elle satisfait donc la propri\'et\'e de densit\'e.
 
  On dira que  la $k$-vari\'et\'e $X$ satisfait la {\it propri\'et\'e de   $R$-densit\'e} si,
 pour toute extension finie de corps $L/k$ et tout $P \in X(L)$,
les points   $Q \in X(L)$ qui sont $R$-\'equivalents \`a $P$ sont
denses dans $X_{L}$ pour la topologie de Zariski.
Donnons deux classes de telles vari\'et\'es.

\begin{prop}\label{cubiqueRdense} Soit $k$ un corps de caract\'eristique z\'ero.
Toute  $k$-hypersurface cubique lisse $X$ dans $\P^n_{k}$,
avec $n \geq 3$,  satisfait la propri\'et\'e de $R$-densit\'e.
\end{prop} 
  \begin{proof} On se ram\`ene au cas $X(k) \neq \emptyset$.
  Comme $X$ est alors $k$-unirationnelle   \cite{Ko02}, il existe un ouvert non vide  $V \subset \P^{n-1}_{k}$
et un $k$-morphisme g\'en\'eriquement fini dominant $f : V \to X$. Comme $f(V(k))$ est dense dans $X$ pour la topologie de Zariski, il existe $A \in f(V(k))$ distinct de $P$ tel que la droite par $A$ et $P$ d\'ecoupe exactement trois
points rationnels distincts, $A,B,P$ sur $X$. La sym\'etrie $t_{B}$ par rapport \`a $B$
est bien d\'efinie en $A$ et satisfait $t_{B}(A)=P$. Quitte \`a remplacer $V$ par un ouvert non vide $W$,
$t_{B} \circ f  $ d\'efinit   un $k$-morphisme
$g : W \to X$ qui est dominant et satisfait $P \in g(W(k))$.   Comme $W$ est un ouvert de $\P^{n-1}_{k}$,
tout  point de    $g(W(k)) \subset X(k)$  est $R$-\'equivalent \`a $P$  sur $X$.
\end{proof}
  
  L'\'enonc\'e suivant est   d\^{u} \`a J. Koll\'ar \cite[Thm. 1.4, Cor. 1.5, Rem. 1.10]{KoAM}.
  
 \begin{thm}\label{ratconnexeRdense}  Soit $k$ un corps fertile de caract\'eristique z\'ero.
 Si  $X$  est  une $k$-vari\'et\'e projective et lisse g\'eom\'etriquement rationnellement connexe,
alors elle satisfait la propri\'et\'e de $R$-densit\'e.
 \end{thm}
 \begin{proof} Soient $k$ un corps fertile de caract\'eristique z\'ero et $X$ une 
 $k$-vari\'et\'e projective et lisse g\'eom\'etriquement rationnellement connexe.
 Soit $P \in X(k)$. 
 Koll\'ar montre d'abord qu'il existe un $k$-morphisme $f: \P^1_{k}  \to X$
 tel que le fibr\'e vectoriel $f^*T_{X}$ soit ample sur $\P^1_{k}$ 
 et que l'on ait deux $k$-points $A,B$  de $\P^1_{k}$
 avec $f(A)=P$ et $Q:=f(B)\neq P$.

 Il montre ensuite (point 4.2 de \cite{KoAM}) qu'il existe
 une $k$-vari\'et\'e $V$ qui est un ouvert du sch\'ema ${\rm Hom}(\P^1_{k}, X, B \mapsto Q)$
 tel que le morphisme d'\'evaluation 
 $$W=( \P^1_{k}  \setminus B) \times_{k} V \to  X$$
 donn\'e  par $(t, g) \mapsto  g(t)$ soit lisse. 
 
 Le $k$-point  $P$ est  dans l'image de $W(k)$ par cette application,
 et tous les $k$-points de $f(W(k))$ sont $R$-\'equivalents sur $X$, 
 via   cette application, via le point $Q=g(B)$.
 \end{proof}
 
   Etant donn\'es une  vari\'et\'e quasiprojective lisse int\`egre $X$
  sur le corps $k$ (de caract\'eristique z\'ero)
  et un entier naturel $m \geq 1$, on note ${\rm Sym}^mX$ le quotient de
  $X^m$ par l'action du groupe sym\'etrique $\frak{S}_{m}$.
  Il y a une bijection naturelle entre les $k$-points de ${\rm Sym}^mX$
  et les z\'ero-cycles effectifs sur $X$ de degr\'e $m$.
   Le  groupe $\frak{S}_{m}$ agit librement sur le compl\'ementaire
   dans $X^m$ des diagonales partielles. Le quotient 
   de ce compl\'ementaire par  $\frak{S}_{m}$ est un ouvert
   lisse de ${\rm Sym}^mX$ qu'on notera ${\rm Sym}^m_{sep}X$.
Les $k$-points de  ${\rm Sym}^m_{sep}X$ correspondent aux z\'ero-cycles effectifs de la forme
   $\sum_{j} P_{j}$ avec $P_{j}$ des points ferm\'es distincts sur $X$
   (un tel z\'ero-cycle effectif, sans multiplicit\'es, sera appel\'e s\'eparable) dont les corps
   r\'esiduels $k(P_{j})$
   satisfont  $\sum_{j} [k(P_{j}):k]= m$.
   On trouvera une \'etude g\'en\'erale d\'etaill\'ee de cette correspondance
   dans \cite{Ry}.

\begin{prop}\label{Rdensecycles}
Soit $k$ un corps de caract\'eristique z\'ero. Soit $X$
 une $k$-vari\'et\'e projective et lisse, g\'eom\'etriquement connexe.
 Soient $P_{1}, \dots, P_{t}$ des points ferm\'es de degr\'es respectifs
 $s_{1}, \dots, s_{t}$  sur $k$ et soit 
  $z=P_{1}+ \dots + P_{t}$
 le z\'ero-cycle associ\'e sur $X$, qui correspond aussi \`a un $k$-point de
 $W={\rm Sym}^{s_{1}}_{sep} X \times \dots \times {\rm Sym}^{s_{t}}_{sep}X$.
On consid\`ere l'ensemble ${\mathcal E}$ des $k$-points de $W$, de z\'ero-cycle associ\'e
$z_{1}+ \dots+ z_{t}$ tels que  pour chaque $i$  le z\'ero-cycle effectif $z_{i}$, de degr\'e $s_{i}$,
soit  rationnellement \'equivalent \`a $P_{i}$
sur $X$. Si $X$ satisfait la propri\'et\'e de $R$-densit\'e, alors  ${\mathcal E} \subset W(k)$ est dense dans $W$ pour la topologie de Zariski.
 \end{prop}
 
 \begin{proof} Il suffit de le montrer dans le cas $t=1$. Soit donc $s>0$ un entier et
 $P$ un point ferm\'e de degr\'e $s$ sur $X$, de corps r\'esiduel $L$. 
 Le point $P$ d\'efinit un point de $X(L)$.
  Soit $\overline k$ une cl\^oture alg\'ebrique de $k$.
La projection $\pi : X_{L}  \to X$ induit une application $\pi_{*}$ de $X(L)$ dans l'ensemble
des cycles effectifs de degr\'e $s$ sur $X$.   Cette application peut aussi \^etre d\'ecrite
de la mani\`ere suivante.
Soit $P \in X(L)$ et $\{(P_{1}, \dots, P_{s})\}$ l'ensemble de ses images dans $X^s(\overline{k})$
par les divers plongement de 
$L$ dans $\overline{k}$.
L'image de 
 $(P_{1}, \dots, P_{s}) \in X^s({\overline k})$ 
dans  ${\rm Sym}^s X({\overline k})$
est invariante sous l'action du groupe de Galois ${\rm Gal}({\overline k}/k)$. Ceci d\'efinit donc une application
$X(L) \to  {\rm Sym}^s X(k)$, ensemble qui co\"{\i}ncide avec l'ensemble des z\'ero-cycles effectifs
de degr\'e $s$ sur $X$.

Si deux points $P,Q$ de $X(L)$ sont
$R$-\'equivalents sur $X_{L}$, alors les z\'ero-cycles $\pi_{*}(P)$ et $\pi_{*}(Q)$
sont rationnellement \'equivalents sur $X$.
Sous l'hypoth\`ese que  $X$ satisfait la propri\'et\'e de $R$-densit\'e,
l'ensemble des points de $X(L)$  qui sont $R$-\'equivalents \`a $P$ sur $X_{L}$ est dense dans $X_{L}$
pour la topologie de Zariski. Ceci implique 
que l'ensemble ${\mathcal E}$ des $k$-points de $ {\rm Sym}^s X(k)$  correspondant \`a des 
z\'ero-cycles effectifs rationnellement \'equivalents \`a $P$, vus comme z\'ero-cycles de degr\'e $s$
sur $X$, est dense dans  $ {\rm Sym}^s X$ pour la topologie de Zariski.
 \end{proof}
 
Pour la  propri\'et\'e plus faible de densit\'e, on a le r\'esultat suivant, dont la
d\'emons\-tration est identique \`a celle de la proposition \ref{Rdensecycles}.
 
   \begin{prop}\label{densecycles}
Soit $k$ un corps de caract\'eristique z\'ero. Soit $X$
 une $k$-vari\'et\'e projective et lisse, g\'eom\'etriquement connexe.
 Soient $P_{1}, \dots, P_{t}$ des points ferm\'es de degr\'es respectifs
 $s_{1}, \dots, s_{t}$  sur $k$ et soit 
  $z=P_{1}+ \dots + P_{t}$
 le z\'ero-cycle associ\'e sur $X$, qui correspond aussi \`a un $k$-point de
 $W={\rm Sym}^{s_{1}}_{sep} X \times \dots \times {\rm Sym}^{s_{t}}_{sep}X$.
 Si $X$ satisfait la propri\'et\'e de densit\'e, 
  l'ensemble ${\mathcal E}$ des $k$-points de $W$, de z\'ero-cycle associ\'e
$z_{1}+ \dots+ z_{t}$  avec les $z_{i}$ z\'ero-cycles effectifs de degr\'e $s_{i}$,
est dense dans $W$ pour la topologie de Zariski.
 \end{prop}

\subsection{Th\'eor\`eme de Bertini et variantes}

Rappelons l'une des versions des  th\'eo\-r\`emes de Bertini.

\begin{thm}\label{bertini} 
 Soit $k$ un corps de caract\'eristique z\'ero. Soit $X$ une
$k$-vari\'et\'e projective, lisse, g\'eom\'etriquement connexe. Soit 
$f : X \to \P^n_{k}$ un $k$-morphisme d'image de dimension au moins 2
et engendrant l'espace projectif $\P^n_{k}$. Il existe un ouvert non vide de 
l'espace projectif dual de $\P^n_{k}$ tel que, pour tout hyperplan $h$ de $\P^n_{k}$
correspondant \`a un point de cet ouvert, la $k$-vari\'et\'e $X_{h}=f^{-1}(h) \subset X$
soit projective, lisse, g\'eom\'etriquement connexe.
\end{thm}

R\'ef\'erence : Jouanolou \cite[Chap. I, Th\'eor\`eme 6.3]{J}.
Sur un corps alg\'ebriquement clos :  Hartshorne \cite[Cor. III.10.9 et  Ex. III.11.3]{H}

\begin{lemma}\label{klar}  Soit $k$ un corps. Soit $X$ une
$k$-vari\'et\'e projective, lisse, g\'eom\'etri\-que\-ment connexe. Soit 
$f : X \to \P^n_{k}$ un $k$-morphisme  dont l'image engendre l'espace projectif $\P^n_{k}$.
Soit $r \leq n+1$ un entier. Il existe un ouvert de Zariski non vide de $X^r$ dont les points g\'eom\'etriques
sont les $r$-uples  $(P_{1}, \dots, P_{r}) \in X^r$ 
dont les images sont des points projectivement ind\'ependants dans $\P^n$.
\end{lemma}

\begin{proof}
C'est clair.
\end{proof}

\begin{prop}\label{produitsimple}
Soit $k$ un corps de caract\'eristique z\'ero.  Soit $X$ une
$k$-vari\'et\'e projective, lisse, g\'eom\'etri\-que\-ment connexe. Soit 
$f : X \to \P^n_{k}$ un $k$-morphisme d'image de dimension au moins 2,
 engendrant l'espace projectif $\P^n_{k}$.
  Soit $r \leq n$ un entier. Il existe un ouvert  non vide $U \subset X^r$ tel que, pour tout corps $L$
   contenant $k$ et
  pour tout $L$-point   $(P_{1}, \dots, P_{r}) \in U(L)$, il existe un hyperplan $h  \subset \P^n_{L}$ 
    tel que l'image r\'eciproque $X_{h}=f^{-1}(h)  \subset X_{L}$ soit
  une $L$-vari\'et\'e lisse et g\'eom\'etriquement  int\`egre contenant les points  $\{P_{1}, \dots, P_{r}\}$.
  \end{prop}
 
 \begin{proof}
  Soit $d$ la dimension de $X$.
  Notons ${\P}^*$ le projectif des hyperplans de $\P=\P^n$.
  On note indiff\'eremment $h$ un point de $\P^*$ ou l'hyperplan de $\P$
  qu'il d\'efinit.
   Pour $h \in \P^*$, on note 
  $X_{h} = f^{-1}(h)$. Par hypoth\`ese, chaque $X_{h}$ est de dimension 
  $d-1$. Par le th\'eor\`eme \ref{bertini}, il existe un ouvert  non vide $W_{0} \subset \P^*$
  tel que pour tout $h \in W_{0}$, la vari\'et\'e $X_{h}$ soit lisse et g\'eom\'etriquement connexe.

  Soit $Z \subset  X^r \times \P^*$ le ferm\'e dont les points g\'eom\'etriques sont les
   $(P_{1},\dots,P_{r};h)$
  avec $h \in \P^*$ et $P_{i} \in f^{-1}(h)$.
  Soient $p : Z \to \P^*$ et $q : Z \to X^r$ les deux projections.

Soit  $U_{1} \subset X^r$ un ouvert donn\'e par le lemme \ref{klar}.
La restriction $V_{1}=q^{-1}(U_{1}) \to U_{1}$ de  
  $q: Z \to X^r$ au-dessus de $U_{1}$
est une fibration en espaces projectifs de dimension $N-r$. 
La fibre au-dessus d'un point  $(P_{1}, \dots, P_{r}) $ consiste en les hyperplans de $\P^n$ 
qui contiennent $(P_{1}, \dots, P_{r})$.
Cette fibration est localement scind\'ee pour la topologie de Zariski,
localement c'est un espace projectif.
La vari\'et\'e $V_{1}$ est  donc lisse, g\'eom\'etriquement int\`egre, de dimension
$rd+N-r$. 
Au-dessus de tout point $h \in \P^*$, la fibre de la projection
$Z \to \P^*$ est le produit $(X_{h})^r$, qui est de dimension $r(d-1)$.
Si l'image de $V_{1} \subset Z$ via la projection $p : Z  \to \P^*$ n'\'etait pas
Zariski-dense dans $\P^*$, alors la dimension de $V_{1}$
serait au plus $r(d-1)+N-1$.   
Ainsi le morphisme compos\'e
$V_{1} \subset Z \to \P^*$ est dominant. 
Soit $W_{1} \subset \P^*$ un ouvert non vide contenu
dans son image. Soit $W=W_{0} \cap W_{1} \subset \P^*$.
Soit $V = p^{-1}(W) \cap V_{1} \subset Z$. 
Soit $U:= q(V) \subset   X^r$. C'est un ouvert de $U_{1} \subset X^r$,
puisque $q: V_{1}  \to U_{1}$ est un fibr\'e projectif, en particulier est lisse.
Comme  $q: V_{1}  \to U_{1}$ est un fibr\'e projectif localement scind\'e
pour la topologie de Zariski, 
et que le corps de base $k$ est infini, pour tout corps $L$
contenant $k$, la fl\`eche induite $V(L) \to U(L)$ est surjective.

Pour tout point   $h \in W \subset \P^*$,
l'image r\'eciproque $ V_{h}$ via $V \to W$ est
non vide, et c'est un ouvert de $p^{-1}(h) \subset X^r$.
Par ailleurs $p^{-1}(h) = (X_{h})^r$, qui est lisse et  g\'eom\'etriquement connexe
car on  a $W \subset W_{0}$.

 On a bien montr\'e : Pour tout $L$-point $M=(P_{1},\dots,P_{r})$ de $U$, il existe  un $L$-hyperplan $h$ 
 de $\P^n_{L}$, contenant chacun des $f(P_{i})$,  et tel que $f^{-1}(h)  \subset X_{L}$ soit une $L$-hypersurface
 lisse et g\'eom\'etriquement int\`egre.
\end{proof}

   \bigskip
   
  La proposition \ref{produitsimple} admet la g\'en\'eralisation suivante.
    
   \begin{prop}\label{produitsymetrique}
Soit $k$ un corps de caract\'eristique z\'ero.  Soit $X$ une
$k$-vari\'et\'e projective, lisse, g\'eom\'etriquement connexe. Soit 
$f : X \to \P^n_{k}$ un $k$-morphisme d'image de dimension au moins 2,
 engendrant l'espace projectif $\P^n_{k}$.
 Soient $s_{1}, \dots, s_{t}$ des entiers naturels tels que $\sum_{i}s_{i} \leq n$.
   Il existe un ouvert  lisse non vide $$U  \subset {\rm Sym}^{s_{1}}_{sep}X \times \dots \times {\rm Sym}^{s_{t}}_{sep}X$$ 
   tel que, pour tout corps $L$
   contenant $k$ et tout $L$-point de $U$, correspondant \`a une famille de   z\'ero-cycles
   effectifs s\'eparables $z_{i}$ sur $X_{L}$, avec $z_{i} $ de degr\'e $s_{i}$,
   il existe un hyperplan $h  \subset \P^n_{L}$ 
    tel que l'image r\'eciproque $X_{h}=f^{-1}(h)  \subset X_{L}$ soit
  une $L$-vari\'et\'e lisse et g\'eom\'etriquement  int\`egre contenant les points  
  du support du z\'ero-cycle $\sum_{i}z_{i}$.
  \end{prop}
\begin{proof}
On utilise  la proposition \ref{produitsimple} et les notations de sa d\'emons\-tration.
On introduit le ferm\'e
 $$Z_{1} \subset {\rm Sym}^{s_{1}} X \times \dots \times {\rm Sym}^{s_{t}}X \times \P^*$$
 qui est l'image sch\'ematique de $Z \subset X^r \times \P^*$ par le morphisme fini
$$X^r \times \P^* \to {\rm Sym}^{s_{1}} X \times \dots \times {\rm Sym}^{s_{t}}X \times \P^*.$$
La projection $Z \to  X^r $ se quotiente 
par l'action du groupe fini $G=\frak{S}_{s_{1}} \times \dots \frak{S}_{s_{t}}$,
donnant la projection 
$Z_{1} \to {\rm Sym}^{s_{1}} X \times \dots \times {\rm Sym}^{s_{t}}X$.
On peut supposer que l'ouvert $U_{1} \subset X^r$ dans la proposition pr\'ec\'edente 
est contenu dans le compl\'ementaire des diagonales partielles de $X^r$.
 On a $V_{1} \subset Z$.
Le morphisme $V_{1} \to U_{1}$ d\'efinit un fibr\'e projectif localement trivial sur $U_{1}$
pour la topologie de Zariski,
et cette projection est compatible avec l'action  fid\`ele de $G$ sur $V_{1}$ et $U_{1}$.
Il en r\'esulte 
que le quotient $V_{1}/G \to U_{1}/G$ est un fibr\'e projectif localement
trivial pour la topologie de Zariski sur $U_{1}/G$. Soit $V' \subset V_{1}/G $
l'ouvert  qui est l'image de l'ouvert $V \subset V_{1}$ 
 par la projection $V_{1} \to V_{1}/G$, puis  $U' \subset U_{1}/G$
l'ouvert image de $V'$ par le morphisme $V_{1}/G \to U_{1}/G$. Il r\'esulte de ce qui pr\'ec\`ede
que, pour tout corps $L$ contenant $k$, la fl\`eche induite $V'(L) \to U'(L)$ est surjective.
Tout point g\'eom\'etrique de $\P^*$ qui est dans l'image de $V' \subset Z_{1}$ par la projection
$Z_{1}\to  \P^*$ est dans l'image de $V$, et donc correspond \`a un hyperplan dont l'intersection
avec $X$ est lisse et connexe. L'ouvert $U'$ convient pour l'\'enonc\'e de la proposition.
 \end{proof}

\begin{thm}\label{fertile}
Soit $F$ un corps de caract\'eristique z\'ero.  Soit $X$ une
$F$-vari\'et\'e projective, lisse, g\'eom\'etriquement connexe. Soit 
$f : X \to \P^n_{F}$ un $k$-morphisme d'image de dimension au moins 2,
 engendrant l'espace projectif $\P^n_{F}$.
  Soient $P_{1}, \dots, P_{t}$ 
  des points ferm\'es de $X$  de degr\'es respectifs $s_{i}$ sur 
  $k$, tels que la somme des $s_{i}$
  soit au plus \'egale \`a $n$. 

(a)  Si  $X$ satisfait la propri\'et\'e de densit\'e, par exemple si  le corps $F$ est  fertile, 
alors il  existe un hyperplan $h \subset \P^n_{F}$ 
 d\'efini sur $F$ tel que
$X_{h}=f^{-1}(h) \subset X$ soit lisse, g\'eom\'etriquement int\`egre,  et  contienne
    des z\'ero-cycles effectifs $z_{1}, \dots, z_{t}$ de degr\'es respectifs 
  $s_{1}, \dots, s_{t}$.
  
  (b) Si la vari\'et\'e $X$ satisfait la propri\'et\'e de $R$-densit\'e,
  par exemple si  $F$ est  fertile et  
  $X$  est g\'eom\'etriquement rationnellement connexe,
  alors il  existe un hyperplan $h \subset \P^n_{F}$ 
 d\'efini sur $F$ tel que
$X_{h}=f^{-1}(h) \subset X$ soit lisse, g\'eom\'etriquement int\`egre,  et  contienne
    des z\'ero-cycles effectifs $z_{1}, \dots, z_{t}$ de degr\'es respectifs 
  $s_{1}, \dots, s_{t}$, chaque z\'ero-cycle $z_{i}$ \'etant  rationnellement \'equivalent \`a $P_{i}$ sur $X$.

  \end{thm}
  
  \begin{proof}
  Le point (a) est obtenu en combinant les propositions  \ref{densecycles} et \ref{produitsymetrique},
  et en utilisant la d\'efinition des corps fertiles.

  Le point (b)  est  obtenu en combinant 
  les propositions  \ref{Rdensecycles} et  \ref{produitsymetrique}, et   le th\'eor\`eme 
  \ref{ratconnexeRdense} pour les
    vari\'et\'es rationnellement connexes sur un corps fertile.
     \end{proof}
 
\subsection{G\'en\'erisation et sp\'ecialisation} 

 On a l'\'enonc\'e bien connu suivant.
 \begin{lemma}\label{special}
 Soit $R$ un anneau de valuation discr\`ete excellent, $F$ son corps des fractions et $k$ son corps r\'esiduel.
 Soit $\mathcal X$ un $R$-sch\'ema propre. Si la $F$-vari\'et\'e ${\mathcal X} \times_{R}F$ poss\`ede un point ferm\'e $P$ de degr\'e $d$,
 alors il existe un z\'ero-cycle effectif $z$ de degr\'e $d$ sur la $k$-vari\'et\'e ${\mathcal X} \times_{R}k$. 
 \end{lemma}
 \begin{proof} La fermeture int\'egrale de $R$ dans l'extension $F(P)/F$ est un anneau de Dedekind $S$ semi-local,
 fini et plat sur $R$, de degr\'e $d$. Comme le $R$-sch\'ema ${\mathcal X}$ est propre,
 l'adh\'erence du  point $P\in X(F)$  dans  $\mathcal X$ est un sch\'ema fini et plat de degr\'e $d$.
  La fibre de ce point au-dessus de $\Spec(k) \subset \Spec(R)$ est un sous $k$-sch\'ema
de dimension z\'ero de  ${\mathcal X} \times_{R}k$, dont le z\'ero-cycle associ\'e est de degr\'e $d$.
 \end{proof}

 \begin{prop}\label{reducfertile}
 Soient $k$ un corps et $F=k((t))$ le corps des s\'eries formelles sur $k$.
 Soit $X$ une $k$-vari\'et\'e propre.

 (a)  Le pgcd des degr\'es des points ferm\'es a la m\^eme valeur sur
  $X$ et sur $X_{F}$.

 (b)  Pour tout entier $r\geq 1$, 
  le plus petit degr\'e d'un point ferm\'e
 de degr\'e premier \`a $r$, qui est aussi le plus petit degr\'e d'un z\'ero-cycle effectif
 de degr\'e premier \`a $r$, a la m\^{e}me valeur sur $X$ et sur $X_{F}$.
 
  (c) Soit $I$ un ensemble d'entiers naturels. Si le groupe de Chow des
 z\'ero-cycles sur $X_{F}$ est engendr\'e par
 les classes de cycles  effectifs de degr\'e $d \in I$, alors il en est de m\^{e}me
sur $X$.
 
 (d) Soit $d\geq 0$ un entier. Si tout z\'ero-cycle sur $X_{F}$ de degr\'e au moins $d$
 est rationnellement \'equivalent \`a un z\'ero-cycle effectif, alors il en de m\^{e}me
 sur $X$.

  \end{prop}
  \begin{proof}
  Si $P \in X$ est un point ferm\'e de $X$, alors $P\times_{k}F$ est un point ferm\'e de $X_{F}$
  de m\^{e}me degr\'e. Si $M$ est un point ferm\'e de $X_{F}$ de degr\'e $d$, d'apr\`es le lemme \ref{special},
   il existe 
  un z\'ero-cycle sur $X$ de degr\'e $d$, et si $d$ est premier \`a $r$, il existe sur $X$
  un point ferm\'e de degr\'e premier \`a $r$ et au plus \'egal \`a $d$.
 Les \'enonc\'es (c) et (d) sont des cons\'equences de l'existence et des propri\'et\'es de l'homomorphisme
 de sp\'ecialisation sur les groupes de Chow \cite[\S 20.3]{F}.
  \end{proof}

\bigskip

 \section{Surfaces cubiques lisses}\label{toutcubique} 
 
 \subsection{Surfaces cubiques avec un z\'ero-cycle de degr\'e 1}\label{cubique1}

Le th\'eor\`eme suivant est d\^u \`a Coray \cite{C1}. 
Nous en reproduisons les diff\'erents
pas, avec la simplification apport\'ee par l'utilisation du th\'eor\`eme \ref{fertile}(a) : il n'y a plus
de discussion des cas possibles  o\`u les courbes utilis\'ees dans la d\'emonstration
sont r\'eductibles ou singuli\`eres. 
 
  \begin{thm} (Coray) \label{cubiques}
  Soit $k$ un corps de caract\'eristique z\'ero. 
  Si une $k$-surface cubique lisse $X \subset \P^3_{k}$
  contient un z\'ero-cycle de degr\'e 1, alors elle poss\`ede  
  un point ferm\'e de degr\'e 1,  ou 4, ou 10. 
   \end{thm}
   
   \begin{proof} 
   
   L'\'enonc\'e peut se reformuler ainsi : si la $k$-surface cubique lisse
   poss\`ede un point ferm\'e  de degr\'e $d$ premier \`a 3, alors le degr\'e minimal
   d'un tel point est $1$, ou $4$, ou $10$. 
   Notons que ce degr\'e minimal est aussi le
   degr\'e minimal d'un z\'ero-cycle effectif de degr\'e premier \`a 3.
   
   On va syst\'ematiquement appliquer le th\'eor\`eme \ref{fertile}(a).
   On peut le faire soit en invoquant le fait que la propri\'et\'e de densit\'e vaut
   pour les surfaces cubiques lisses sur $k$ car elles sont $k$-unirationnelles
   d\`es qu'elles ont un $k$-point (Segre, \cite{Ko02}), soit en utilisant la
  la proposition \ref{reducfertile} qui permet, pour le th\'eor\`eme \`a d\'emontrer,
  de supposer le corps $k$ fertile.
   On note $K$  le faisceau canonique sur $X$. Le syst\`eme lin\'eaire complet    
    $\lvert -K \rvert$ associ\'e  au faisceau inversible $-K$
d\'efinit le plongement de $X$ dans $\P^3_{k}$.
    Pour tout entier $n>0$,  le syst\`eme lin\'eaire $\lvert -nK \rvert$    
 d\'efinit  un plongement
   dans un espace projectif, d'image de dimension 2, engendrant cet espace projectif.
    Pour un fibr\'e inversible $\L$, on note $h^{i}(X,\L)$, ou $h^{i}(\L)$ quand le contexte est clair,
    la dimension sur $k$
 du groupe de cohomologie coh\'erente $H^{i}(X,\L)$. 
 
 Pour la surface cubique lisse $X$
 comme pour toute surface projective et lisse g\'eom\'etriquement rationnelle, on a 
 $H^1(X,\O_{X})=0$ et $H^2(X,\O_{X})=0$, et donc $\chi(X,\O_{X})=h^0(\O_{X}) - h^1(\O_{X})+h^2(\O_{X})=1$.
 Soit $n\geq 1$. 
 
 Par dualit\'e de Serre \cite[Chap. IV, Prop. 4.1]{AK}
 on a $h^2(-nK)= h^0((n+1)K)$ et $h^1(-nK)=h^1((n+1)K)$.
 On a $h^0((n+1)K) =0$ car $-K$   est ample. 
 
Par dualit\'e de Serre on aussi  On a aussi $h^1((n+1)K)=0$
 par le th\'eor\`eme d'annulation de Kodaira, puisque $-K$ est ample.
   
    Pour $n \geq 1$,   le th\'eor\`eme de Riemann-Roch  sur la surface $X$ (\cite[Chap. IV, \S 8]{Se} \cite[Lecture 12, Prop. 3]{Mu})  donne
      donc $$h^0(-nK)= 3 n(n+1)/2 +1.$$ 
  
 Si $\Gamma$ est une courbe projective,
 lisse, g\'eom\'etriquement connexe dans le syst\`eme lin\'eaire 
$\vert -nK \rvert$, 
 on a
 la formule
 $$g(\Gamma)= p_{a}(\Gamma)= 3 n(n-1)/2 +1.$$
 Une telle courbe contient un z\'ero-cycle de degr\'e $3n= (-K.-nK)$, d\'ecoup\'e par
 un plan de $\P^3_{k}$.   
 
 Soit $d>0$ le degr\'e minimum  d'un z\'ero-cycle effectif  de  degr\'e premier \`a 3 sur  $X$.
 C'est donc aussi le degr\'e minimum  d'un point ferm\'e de  degr\'e premier \`a 3 sur  $X$.
Si $d=1$, on a fini.  
 Supposons $d\geq 2$. Si la  surface cubique poss\`ede un point sur
 une extension quadratique de $k$, une construction bien connue
 montre qu'elle poss\`ede un point rationnel.
 On se limite donc dor\'enavant au cas $d\geq 4$.
 La surface $X$ contient un point ferm\'e de degr\'e $3$, d\'ecoup\'e par
 une droite quelconque de $\P^3_{k}$.
  
 Il existe un unique entier $n \geq 1$ tel que
 $$ g=3 n(n-1)/2 +1 \leq d < 3 n(n+1)/2 +1.$$
 Comme on a $d \geq 4$, on a $n \geq 2$.

 \medskip
 Supposons d'abord $g=3 n(n-1)/2 +1  <  d  \leq  3 n(n+1)/2 -3$.
 Comme $d$ est premier \`a 3 et $d+3  \leq  3 n(n+1)/2$,
 le th\'eor\`eme \ref{fertile}(a) assure l'existence d'une courbe 
 $\Gamma$ projective, lisse, g\'eom\'etriquement connexe
 dans le syst\`eme lin\'eaire $\lvert -nK \rvert$, contenant la r\'eunion
 d'un 
 z\'ero-cycle effectif
 de degr\'e $d$ et d'un
 z\'ero-cycle effectif
  de degr\'e $3$, donc contenant
 un z\'ero-cycle de degr\'e 1, et donc aussi un z\'ero-cycle de degr\'e
 $g=3 n(n-1)/2 +1$.  Par le th\'eor\`eme de Riemann-Roch,
la courbe $\Gamma$ poss\`ede donc un
 z\'ero-cycle effectif de degr\'e $3 n(n-1)/2 +1 <d$,
 ce qui contredit l'hypoth\`ese que $d$ est minimal.
  
\medskip
 
Il reste  donc les possibilit\'es suivantes :
 $$d= 3 n(n+1)/2,$$
 $$d =3 n(n+1)/2 - 1,$$
 $$d  = 3 n(n+1)/2 -2,$$
 $$d= 3n (n-1)/2+1.$$

 Le cas $d= 3 n(n+1)/2$ est exclu, car $d$ est premier \`a $3$.
 
 Dans chacun des trois  autres cas, toute courbe lisse $\Gamma$ 
 dans le syst\`eme lin\'eaire $\lvert -nK \rvert$ contenant un 
 z\'ero-cycle   de degr\'e $d$
 contient un z\'ero-cycle de degr\'e 4, car, comme
 on l'a d\'ej\`a indiqu\'e, elle contient un z\'ero-cycle de degr\'e $3n$.
 
\bigskip

Supposons  $d =3 n(n+1)/2 - 1$. Par le th\'eor\`eme \ref{fertile}(a),
 il existe une courbe $\Gamma$ lisse g\'eom\'etriquement connexe dans le syst\`eme lin\'eaire $\lvert -nK \rvert$ contenant
un z\'ero-cycle effectif
 de degr\'e $d$, degr\'e qui est congru  \`a $2$ mod. 3.
 Comme la courbe $\Gamma$ contient un z\'ero-cycle de degr\'e $4$,
elle contient donc aussi un z\'ero-cycle de degr\'e
$d-4$, degr\'e qui est  premier \`a $3$.
Comme on a $n \geq 2$, on a
$$g= 3 n(n-1)/2 + 1\leq 3 n(n+1)/2 - 1 -4 =d-4.$$
Le th\'eor\`eme de Riemann-Roch sur la courbe $\Gamma$
assure alors l'existence d'un z\'ero-cycle effectif de degr\'e $d-4$,
premier \`a $3$,  ce qui est en contradiction avec l'hypoth\`ese $d$ minimal.

\bigskip

Supposons  $d =3 n(n+1)/2 - 2$ et $n$ impair.
Par le th\'eor\`eme \ref{fertile}(a),  il existe une courbe $\Gamma$ lisse g\'eom\'etriquement connexe
dans le syst\`eme lin\'eaire 
$\lvert -nK \rvert$ contenant
 un  z\'ero-cycle effectif de degr\'e $d$, degr\'e qui est congru  \`a $1$ mod. 3.
 Comme $2$ est combinaison lin\'eaire de $ 3 n(n+1)/2 - 2$ et $3n$,
 il existe alors un z\'ero-cycle de degr\'e 2 sur $\Gamma$.
 La courbe $\Gamma$ 
   contient donc un z\'ero-cycle de degr\'e $d-2$. 
Comme on a $n \geq 2$, 
on a $$g=3 n(n-1)/2 + 1\leq 3 n(n+1)/2 - 2 -2 =d-2.$$
Par le th\'eor\`eme de Riemann-Roch, sur la courbe $\Gamma$, il existe
un z\'ero-cycle effectif de degr\'e $d-2$, qui est   premier \`a 3. 
Ainsi  $X$ poss\`ede un z\'ero-cycle effectif de degr\'e 
 $d-2$ premier \`a 3, ce qui est en contradiction avec l'hypoth\`ese $d$ minimal.

\medskip

Supposons donc $d =3 n(n+1)/2 - 2$ et $n$ pair, donc $n \geq  2$.
Par le th\'eor\`eme \ref{fertile}(a),  il existe une courbe $\Gamma$ lisse dans le syst\`eme lin\'eaire 
$\lvert -nK \rvert$ contenant
 un z\'ero-cycle effectif
 de degr\'e $d$, degr\'e qui est congru  \`a $1$ mod. 3.
 
Dans le cas $n=2$, on a $g=4$ et $d=7$. Comme   $\Gamma$
contient un z\'ero-cycle de degr\'e 4, le th\'eor\`eme de Riemann-Roch sur une courbe montre
l'existence d'un z\'ero-cycle effectif de degr\'e 4 sur une telle courbe, et donc
aussi sur $X$, ce qui est en contradiction avec l'hypoth\`ese $d$ minimal.

On peut donc supposer $n$ pair, $n \geq 4$. Dans ce cas, on a
$$g= 3 n(n-1)/2 + 1 \leq 3 n(n+1)/2 -2 -8=d-8.$$
Comme   $\Gamma$
contient un z\'ero-cycle de degr\'e 4, le th\'eor\`eme de Riemann-Roch sur une courbe montre
l'existence d'un z\'ero-cycle effectif 
 de degr\'e $d-8$ sur $\Gamma$ et donc sur $X$, et $d-8$ est
congru \`a $2$ modulo 3, ce qui est en contradiction avec l'hypoth\`ese $d$ minimal.
 
\bigskip

Il reste \`a examiner le cas  $d= 3n (n-1)/2+1 $, o\`u l'on a $n \geq 2$ et $d\geq 4$.

On a donc une $k$-surface $X$  avec un point ferm\'e $P$ de degr\'e $d$ premier
\`a 3 minimal, au moins \'egal \`a 4. L'unique entier $n$ tel que
 $$ 3 n(n-1)/2 +1 \leq d < 3 n(n+1)/2 +1$$
 satisfait $3 n(n-1)/2 +1=d$.

On prend un point ferm\'e $M$ de degr\'e 3 sur $X$ 
d\'ecoup\'e par une droite $D$  d\'efinie sur $k$, qu'on peut choisir g\'en\'erale car le corps $k$
est infini.
 Soit $p : Y \to X$ l'\'eclatement de $X$ en le point ferm\'e $M$.
  On note $E \subset Y$ le diviseur exceptionnel et $K$ le faisceau canonique sur $X$.
  
  Le syst\`eme lin\'eaire 
  $\lvert p^*(-K)-E \rvert$
  d\'efinit un morphisme $Y \to  D=\P^1_{k}$, dont les
  fibres sont les sections de $X$ par les plans contenant $D$.
  On a un plongement $Y \subset \P^1_{k} \times \P^3_{k}$ dont la projection sur 
   le premier facteur  est d\'efinie par le syst\`eme lin\'eaire $\lvert p^*(-K)-E \rvert$ et la projection sur 
   le second facteur  est d\'efinie  sur le second facteur
  par 
  $\lvert p^*(-K)  \rvert$.
  Il s'en suit que pour tout couple d'entiers $a \geq 1, b\geq 1$
  le faisceau inversible $a(p^*(-K)-E)  + b p^*(-K)$ est tr\`es ample.
   En particulier, pour $n \geq 3$, le faisceau $p^*(-nK)-2E$ est tr\`es ample.
 Le fait que ces faisceaux inversibles soient tr\`es amples peut aussi s'\'etablir 
 en utilisant \cite[Thm. 1]{R}. 
 
  On consid\`ere sur $Y$ les syst\`emes lin\'eaires   
$\lvert p^*(-nK)-2E \rvert$ pour $n \geq 1$. Ceci correspond aux sections de $X$
par des  surfaces de degr\'e $n\geq 3$ dans $\P^3$, 
 avec une singularit\'e au point ferm\'e $M$, qui est de degr\'e 3.
Imposer une telle singularit\'e correspond \`a 9 conditions lin\'eaires.

\begin{lemma}\label{calculRR3}
Soit $n \geq 3$. On a $h^0(Y, p^*(-nK)-2E) =3 n(n+1)/2 -8$,
le  syst\`eme lin\'eaire 
$\lvert p^*(-nK)-2E \rvert$
 d\'efinit un plongement de la surface $Y$ dans un espace projectif 
de dimension    $3 n(n+1)/2 -9$.  Toute courbe g\'eom\'etriquement connexe et lisse $\Gamma$ 
dans le syst\`eme lin\'eaire associ\'e satisfait
$g(\Gamma)=p_{a}(\Gamma)=    3 n(n-1)/2 -2.$
\end{lemma}
\begin{proof} 
Le faisceau canonique $K_{Y}$ sur $Y$ est $p^*(K)+E$.
Par dualit\'e de Serre, on a 
$$h^2(p^*(-nK)-2E)= h^0(p^*(K)+E + p^*(nK)+2E)= h^0(p^*((n+1)K) +3E)$$
et 
$$h^1(p^*(-nK)-2E)= h^1( p^*((n+1)K) +3E).$$
 L'oppos\'e de $p^*((n+1)K) +3E$ est $p^*(-(n+1)K) -3E$
 qui est la somme de $3(p^*(-K)-E) $  et de $p^*(-mK)$ avec $m \geq 1$,
 et donc est tr\`es ample. Ceci implique d'une part
$ h^0(p^*((n+1)K) +3E)=0$, d'autre part d'apr\`es 
le th\'eor\`eme d'annulation de Kodaira,
  $h^1(p^*((n+1)K) +3E)=0$.   
 En utilisant le th\'eor\`eme de Riemann-Roch sur la surface $Y$, ceci donne
$$h^0(Y, p^*(-nK)-2E) = 3 n(n+1)/2 -8.$$  
La formule $p_{a}(\Gamma)= (\Gamma. \Gamma + K_{Y})/2 +1$ donne 
le calcul du genre de $\Gamma$.
\end{proof}
  
 Pour appliquer le  th\'eor\`eme \ref{fertile}(a), on a  besoin de l'in\'egalit\'e
$$3 n(n-1)/2+1=  d \leq  3 n(n+1)/2 - 9$$
soit $n\geq 20/6$ et donc $n > 3$.
On se restreint donc maintenant  \`a $n \geq 4$.
Comme on~a $d=3n (n-1)/2+1$, ceci \'equivaut \`a
ignorer les cas $d=1$, $d=4$ et $d=10$.

 Le th\'eor\`eme \ref{fertile}(a)   assure
 l'existence sur $Y$ d'une $k$-courbe $\Gamma$
  lisse et g\'eom\'e\-tri\-quement connexe  sur $Y$,
 de genre $g = 3 n(n-1)/2 -2$, contenant un z\'ero-cycle effectif de degr\'e 
$d=3n (n-1)/2+1$.

La courbe $\Gamma$ contient aussi un z\'ero-cycle de degr\'e $3n$,
d\'ecoup\'e par l'image r\'eciproque d'une section plane de $X\subset \P^3_{k}$.
La courbe $\Gamma$ poss\`ede donc un z\'ero-cycle de degr\'e 2.
Elle contient donc un z\'ero-cycle de degr\'e 
$d-2= 3n (n-1)/2-1$, de degr\'e premier \`a 3,
et satisfaisant $d-2 \geq g$.
 Le th\'eor\`eme de Riemann-Roch sur une courbe
assure  qu'il existe sur $\Gamma$,  et donc sur $Y$, et donc sur $X$,
  un z\'ero-cycle effectif
de degr\'e $d-2$, premier \`a $3$, ce qui contredit
l'hypoth\`ese $d$ minimal.

\medskip

On voit donc que l'on a soit $d=1$, soit $d=4$, soit $d=10$.
\end{proof}

\subsection{Surfaces cubiques avec un point rationnel}\label{dp3pointrat}

\begin{thm}\label{soustractionpointdp3} Soit $k$ un corps de caract\'eristique z\'ero. 
Soit $X \subset \P^3_{k}$ une surface cubique lisse poss\'edant un
point rationnel.

(a) Soit $Q\in X(k)$ un point rationnel.
Tout z\'ero-cycle effectif de degr\'e au moins  3 sur $X$
est rationnellement  \'equivalent \`a un z\'ero-cycle
effectif $z_{1}+rQ$ avec $r \geq 0$ et $z_{1}$
effectif de degr\'e au plus 3.

(b) Tout z\'ero-cycle de degr\'e positif ou nul est rationnellement \'equivalent
\`a une diff\'erence $z_{1}-z_{2}$ avec $z_{1}$  effectif
et  $z_{2}$ efffectif de degr\'e au plus 3. 

(c) Tout z\'ero-cycle de degr\'e z\'ero
est  rationnellement \'equivalent \`a la diff\'erence de deux cycles effectifs de degr\'e 3. 

(d) Tout z\'ero-cycle de degr\'e au moins 3 est rationnellement \'equivalent
\`a un z\'ero-cycle effectif ou \`a la diff\'erence d'un z\'ero-cycle effectif
et d'un point ferm\'e de degr\'e 3.

(e) Le groupe de Chow des z\'ero-cycles  sur $X$ est engendr\'e
par les classes des points rationnels et des points ferm\'es de degr\'e 3.

(f) Tout z\'ero-cycle sur $X$ de degr\'e au moins \'egal \`a 10 est rationnellement
\'equivalent \`a un z\'ero-cycle effectif.

\end{thm}

\begin{proof}
On va syst\'ematiquement appliquer le th\'eor\`eme \ref{fertile}(b).
   
   On peut le faire car la propri\'et\'e de $R$-densit\'e vaut
   pour les surfaces cubiques lisses sur tout $k$ de caract\'eristique z\'ero
    (Proposition \ref{cubiqueRdense}).
   
   On pourrait aussi observer que d'apr\`es la proposition  \ref{reducfertile}, 
   pour le th\'eor\`eme
   \`a d\'emontrer, on peut supposer le corps $k$ fertile, ensuite invoquer le fait bien
   connu qu'une surface cubique lisse est g\'eom\'etriquement rationnelle et donc
   g\'eom\'e\-tri\-quement rationnellement connexe, et enfin  appliquer
  le th\'eor\`eme \ref{ratconnexeRdense}.  Cette m\'ethode 
  sera   utile dans l'\'etude des surfaces de del Pezzo de degr\'e 2
  et de degr\'e~1.

 Soit $z$ un z\'ero-cycle effectif  de degr\'e $d\geq 1$.
Soit $n$ le plus petit entier tel que
$d+2 \leq 3n(n+1)/2+1$.
On a donc 
$$ 3n(n-1)/2+1 < d+2$$
soit encore $$ 3n(n-1)/2 \leq d.$$

D'apr\`es le th\'eor\`eme \ref{fertile}(b),
quitte \`a remplacer $z$ par un z\'ero-cycle effectif rationnellement \'equivalent
encore not\'e $z$
et $Q$ par un point rationnel rationnellement \'equivalent encore not\'e $Q$, on peut supposer qu'il existe
 une courbe lisse  g\'eom\'etriquement connexe $\Gamma$ dans le syst\`eme lin\'eaire 
$\lvert -nK \rvert$
contenant  le  z\'ero-cycle  $z$  et  le point rationnel $Q$.
 
On a $g(\Gamma)= p_{a}(\Gamma)=  3n(n-1)/2 +1$.
Si l'on a 
$$  d-1 \geq  3n(n-1)/2 +1,$$
alors le z\'ero-cycle $z-Q$ est rationnellement \'equivalent sur $\Gamma$, donc sur $X$,
\`a un z\'ero-cycle effectif de degr\'e $d-1$.

La condition est satisfaite  sauf si
$$ 3n(n-1)/2 \leq d \leq 3n(n-1)/2 +1.$$

\medskip

Consid\'erons le cas $d =  3n(n-1)/2 +1.$ 
Ici $p_{a}(\Gamma)= d$.
Dans ce cas, on fixe un autre point rationnel $R \in X(k)$, distinct de $Q$,
 non dans le support de $z$, et non situ\'e sur une des droites de $X$,
et on exige
 $$  d+1 +3 +1 \leq   3n(n+1)/2+1.$$
 Ceci est possible si 
 $$ 3n(n-1)/2 + 6 \leq 3n(n+1)/2+1$$
 soit encore $n \geq 2$.

 On consid\`ere l'\'eclatement $p : Y \to X$ en le point $R$, 
 la courbe exceptionnelle
 $E \subset Y$, et le faisceau inversible
 $p^*(-nK)-2E$ sur $Y$.
 La surface $Y$ est une surface de del Pezzo de degr\'e 2. 
 Le faisceau anticanonique sur $Y$ est donn\'e par  $p^*(-K)-E$.
 Il est ample, son double  $p^*(-2K)-2E$ est tr\`es ample.
Le syst\`eme lin\'eaire $\lvert p^*(-K) \rvert$ sur $Y$ correspond
au morphisme $Y \to X$, ceci implique que pour tous entiers $a\geq 0$ et $b \geq 1$,
le faisceau inversible $ap^*(-K)+b(p^*(-K)-E)$ est ample, et que
le faisceau inversible   $ap^*(-K)+2b(p^*(-K)-E)$ est tr\`es ample.
On peut aussi \'etablir ces divers \'enonc\'es de tr\`es-amplitude
 par une application de \cite[Thm. 1]{R}.

Sur la surface $Y$, le th\'eor\`eme de Riemann-Roch pour  le faisceau
$$L=p^*(-nK)-2E= -nK_{Y}+(n-2)E,$$
le th\'eor\`eme de dualit\'e de Serre et  le th\'eor\`eme d'annulation de Kodaira
donnent alors, pour $n\geq 2$,
$$ h^0(Y, p^*(-nK)-2E)=  3n(n+1)/2 -2.$$
Pour $n \geq 2$, le syst\`eme lin\'eaire  $\lvert p^*(-nK)-2E \rvert$ d\'efinit donc un plongement
de  la surface $Y$ dans un espace projectif $\P^N$ avec $N=3n(n+1)/2 -3$,
espace projectif qu'elle engendre. Comme on a $d+1 =  3n(n-1)/2 +2 \leq 3n(n+1)/2 -3$,
le th\'eor\`eme \ref{fertile}(b) assure l'existence dans le syst\`eme lin\'eaire  $\lvert p^*(-nK)-2E \rvert$
d'une courbe $\Gamma \subset Y$ projective, lisse
et g\'eom\'etriquement int\`egre,  et qui contient un z\'ero-cycle effectif $z_{1}$ rationnellement \'equivalent
\`a  $p^*(z)$ sur $Y$  et un point rationnel  $Q_{1}$ rationnellement \'equivalent au point 
  $p^*(Q)$. Le genre de cette courbe est 
$3n(n-1)/2$. 
 Le z\'ero-cycle $z_{1}-Q_{1}$ est de degr\'e $3n(n-1)/2$. Il est donc rationnellement \'equivalent,
 sur $\Gamma$,  et donc sur $Y$, \`a un z\'ero-cycle effectif de degr\'e $d-1$. 
 Le  z\'ero-cycle $p^*(z) -p^*(Q)$ sur $Y$ est donc rationnellement \'equivalent \`a 
 un z\'ero-cycle effectif, et il en est donc de m\^eme de son image directe $z-Q$
 sur $X$.
  
\medskip

Consid\'erons le cas $d =  3n(n-1)/2$ et $p_{a}=d+1$.
On s'int\'eresse au cas $d\geq 4$ et donc $n\geq 3$.

Dans ce cas on va fixer un couple de points rationnels $R$ et $S$ suffisamment g\'en\'eral,
et imposer un point double en chacun de ces points, ce qui impose 6 conditions lin\'eaires
pour le syst\`eme lin\'eaire  $\lvert -nK \rvert$.  Voici comment faire cela formellement.

Soit $p: Y \to X$ l'\'eclat\'e de $X$ en $R$ et $S$, et soient $E_{R} \subset Y$, resp. $E_{S} \subset Y$
les courbes exceptionnelles.  La surface $Y$ est une surface de del Pezzo de degr\'e 1.

Sur cette surface, le faisceau inversible $-K_{Y}= p^*(-K)-E_{R}-E_{S}$ est ample et
 le faisceau  inversible $-3K_{Y}$ est tr\`es ample \cite[Chap. III, Prop. 3.4]{Kol}.

Pour $n \geq 3$,  le faisceau inversible
$$p^*(-nK)-2E_{R}-2E_{S}
=-nK_{Y}+(n-2)E_{R} +(n-2)E_{S}$$
 sur $Y$ est tr\`es ample, comme on voit en utilisant \cite[Thm. 1]{R}.

  En utilisant le th\'eor\`eme de Riemann-Roch   pour le faisceau $p^*(-nK)-2E_{R}-2E_{S}$
  sur $Y$,
   la dualit\'e de Serre
  et le th\'eor\`eme d'annulation de Kodaira, pour $n \geq 3$ on obtient 
  $$h^0(Y, p^*(-nK)-2E_{R}-2E_{S})= 3n(n+1)/2-5.$$

  Pour $n \geq 3$, le syst\`eme lin\'eaire  $\lvert p^*(-nK)-2E_{R}-2E_{S} \rvert$
d\'efinit un plongement de $Y$ dans $\P^N$ avec $N=  3n(n+1)/2-6$,
dont l'image engendre projectivement $\P^N$.

On a
$$ d+1  +1 \leq 3n(n+1)/2-5 $$
c'est-\`a-dire
$$ 3n(n-1)/2   \leq  3n(n+1)/2 -7,$$
puisque l'on a  $n \geq 3$.

D'apr\`es le th\'eor\`eme \ref{fertile}(b), il existe 
 un z\'ero-cycle effectif $z'$  sur $Y$ rationnellement \'equivalent \`a $p^*(z)$ sur $Y$, 
un point rationnel $Q' \in Y(k)$ rationnellement \'equivalent \`a $p^*(Q)$  sur $Y$
et   une courbe $\Gamma \subset Y$ g\'eom\'etriquement int\`egre et lisse
sur $Y$ dans le syst\`eme lin\'eaire  $\vert p^*(-nK)-2E_{R}-2E_{S} \rvert$   qui
  contient le support de $z'$ et  le point $Q'$.
  Le genre de cette courbe est $d-1$, et le z\'ero-cycle $z'-Q'$
 est donc rationnellement \'equivalent sur $\Gamma$  \`a
 un z\'ero-cycle effectif, il en est donc de m\^{e}me pour $z-Q$ sur $X$.
 
 En conclusion, tout z\'ero-cycle $z$ effectif de degr\'e $d$ au moins \'egal \`a 4  sur $X$ est
 rationnellement \'equivalent \`a un z\'ero-cycle $z_{1}+ rQ$,
 avec $z_{1}$ effectif  de degr\'e au plus~3.
 
 Ceci \'etablit le point (a).  On notera que le choix du point rationnel $Q$
 est arbitraire.
 Les points (b) et (c) sont des cons\'equences \'evidentes de (a).
 
 Il y a une classe standard  $\ell$ dans $CH_{0}(X)$ de degr\'e 3,
celle d\'ecoup\'ee par une droite d\'efinie sur $k$
quelconque mais non situ\'ee sur la surface $X$. 
Comme $X$ poss\`ede des points rationnels, et que ces points sont
denses pour la topologie de Zariski,
on peut trouver une telle droite qui d\'ecoupe sur $X$ trois points rationnels distincts.
Si $P$ est un point ferm\'e de degr\'e 2 non situ\'e sur une droite de la surface,
alors la droite qu'elle d\'efinit d\'ecoupe sur $X$ une somme $P+p$ avec
$p$ point rationnel, et $P+p$ est dans la classe $\ell$,
donc \'equivalent \`a la somme de trois points rationnels align\'es.
Si $P$ est situ\'e sur une droite $D$ de la surface, alors 
$P$ est rationnellement \'equivalent sur $D$ donc sur $X$
\`a $2Q$ pour tout point rationnel $Q$ de la droite.
En r\'esum\'e, tout point ferm\'e $P$ de degr\'e 2 sur $X$
est rationnellement \'equivalent \`a un z\'ero-cycle
$a+b+c-d$ avec $a,b,c,d$ points rationnels.
  Les  r\'esultats   (d) et  (e)  s'obtiennent alors  formellement \`a partir de (a).

 D\'emontrons (f). 
  D'apr\`es la proposition \ref{reducfertile}, on peut supposer $k$ fertile.
  Le plus petit entier $d$ pour lequel il existe un entier naturel $n$
 avec
 $ 3n(n-1)/2+1 \leq  d-3$ et $d+3+1  \leq  3n(n+1)/2 +1$
 est $d=13$, qui correspond \`a $n=3$.
 
Consid\'erons  un z\'ero-cycle $z-P$ avec $z$ effectif de degr\'e $d=13$ et $P$ un
point ferm\'e de degr\'e 3. Le th\'eor\`eme \ref{fertile}(b)   montre l'existence
d'un z\'ero-cycle effectif $z'$ rationnellement \'equivalent \`a $z$, d'un
z\'ero-cycle effectif $P'$ de degr\'e 3 rationnellement \'equivalent \`a $P$,
et d'une courbe lisse g\'eom\'etriquement int\`egre $\Gamma \subset X$
dans le syst\`eme lin\'eaire $\lvert -3K \rvert$  de genre
$g= 10$ contenant le support de $z'$ et celui de $Q'$.
Le th\'eor\`eme de Riemann-Roch sur $\Gamma$ assure alors l'existence d'un z\'ero-cycle 
effectif de degr\'e $10$  rationnellement \'equivalent sur $\Gamma$, donc sur $X$,
\`a $z-P$.

Soit  $z$ un z\'ero-cycle quelconque sur $X$ de degr\'e au moins 10.
D'apr\`es (d), soit il est rationnellement \'equivalent
 \`a un z\'ero-cycle effectif, soit il est 
rationnellement \'equivalent \`a une diff\'erence $z_{1}-P$ avec $P$ point ferm\'e
de degr\'e 3 et $z_{1}$  z\'ero-cycle effectif de degr\'e au moins $13$.
D'apr\`es (a), le z\'ero-cycle $z_{1}$ est rationnellement \'equivalent
\`a $z_{2}+rQ$ avec $Q$ point rationnel, $r\geq 0$, et $z_{2}$ z\'ero-cycle
effectif de degr\'e 13. Ainsi $z$ est rationnellement \'equivalent \`a
$rQ+z_{1}-P$ avec $z_{1}$ effectif de degr\'e 13. Et on a vu ci-dessus que, pour un tel $z_{1}$, le z\'ero-cycle $z_{1}-P$
est  rationnellement \'equivalent \`a un z\'ero-cycle effectif.
  \end{proof}

  \begin{rmk}
 Dans \cite{CTC}, pour une surface fibr\'ee en coniques relativement minimale 
 au-dessus de la droite $\P^1_{k}$,    notant $r$ le nombre de fibres
g\'eom\'etriques singuli\`eres de la fibration $X\to \P^1_{k}$,
nous montrons que tout z\'ero-cycle sur $X$ de degr\'e au moins 
${\rm max}(0,  \lfloor{r/2}\rfloor -1)$ est rationnellement \'equivalent \`a un z\'ero-cycle effectif.
Une autre d\'emonstration, plus conceptuelle, fut plus tard obtenue par P. Salberger
\cite{S}. La d\'emonstration de \cite{CTC}  requiert des discussions sur la d\'ecomposition possible
des courbes obtenues dans un syst\`eme lin\'eaire. Il n'est pas clair si 
on pourrait utiliser la m\'ethode du \S \ref{bertinifertile}   pour simplifier cette
d\'emonstration.  
 \end{rmk}
 
 \begin{rmk}
L'analogue du th\'eor\`eme \ref{soustractionpointdp3} est connu pour les surfaces 
de del Pezzo $X$  de degr\'e 4 avec un point rationnel. Dans ce cas on a mieux.
Par \'eclatement
d'un $k$-point non situ\'e sur les droites de $X$, on obtient une surface
cubique $Y$ fibr\'ee en coniques au-dessus de $\P^1_{k}$, avec 5
fibres g\'eom\'etriques d\'eg\'en\'er\'ees. Le th\'eor\`eme de  \cite{CTC} donne alors
 que tout z\'ero-cycle sur $Y$ de degr\'e au moins 1 est rationnellement
\'equivalent \`a un z\'ero-cycle effectif. Ceci vaut donc aussi pour une surface de del Pezzo $X$ de degr\'e 4
poss\'edant un point rationnel 
(l'existence d'un tel point suffit pour que les point rationnels soient denses pour la topologie
de Zariski sur $X$).
 \end{rmk}

\section{Surfaces de del Pezzo de degr\'e 2}\label{dp2}
  
\subsection{Surfaces de del Pezzo de degr\'e 2 avec un z\'ero-cycle de degr\'e 1}\label{137}

On suit la m\'ethode de Coray pour les surfaces cubiques  \cite{C1},
 avec la flexibilit\'e donn\'ee par le th\'eor\`eme
\ref{fertile}(a).

  \begin{thm}\label{DP2}
  Soient $k$ un corps de caract\'eristique z\'ero et $X$ une $k$-surface de del Pezzo de degr\'e 2.
Si $X$ poss\`ede un z\'ero-cycle de degr\'e 1, elle poss\`ede un point ferm\'e de degr\'e $1$, ou $3$, ou $7$.
 \end{thm}

\begin{proof}

Une telle surface $X$ poss\`ede des points dans des extensions quadratiques
du corps de base $k$,  puisque c'est un rev\^etement double de ${\bf P}^2_{k}$,
donn\'e par le syst\`eme lin\'eaire  associ\'e \`a $-K$.
Soit $Q$ un point de degr\'e 2 sur $X$. Supposons donn\'e un point ferm\'e de degr\'e
$d$ impair. On peut supposer $d$ minimal avec cette propri\'et\'e.
Si $d=1$, on a un point rationnel. Supposons donc $d \geq 3$.

 D'apr\`es la proposition \ref{reducfertile}, on peut supposer le corps $k$ fertile.
 Les seuls faisceaux inversibles  \'evidents sur $X$ sont les faisceaux 
$-nK$.
Ils sont
amples pour $n \geq 1$, et tr\`es amples pour $n \geq 2$, par exemple par \cite[Thm. 1]{R}.
 Soit $n \geq 2$. 
 On a 
  $$h^2(nK)=h^0((1+n)K)=0,$$ car le faisceau inversible $-K$ est ample.
  Comme $-K$ est ample, on a $$h^1(nK)= h^1((1+n)K)=0$$ d'apr\`es le th\'eor\`eme d'annulation de Kodaira.
  Le th\'eor\`eme de Riemann-Roch sur la surface $X$ donne alors
$$h^0(-nK) = (-nK. (-nK-K))/2 + 1 = n^2+n+1.$$
Pour tout entier $n\geq 1$, le faisceau inversible  $-nK$ est ample et  ses sections d\'efinissent un morphisme
$X \to \P^{n^2+n}_{k}$ d'image de dimension au moins 2,  engendrant projectivement 
$\P^{n^2+n}_{k}$.

Pour $\Gamma$ une courbe projective et lisse dans le syt\`eme lin\'eaire $\lvert -nK \rvert$, on a
$$g(\Gamma)=p_{a}(\Gamma)= (-nK. -nK+K)/2+1= n^2-n+1.$$
Une telle courbe $\Gamma$ contient un z\'ero-cycle (effectif) de degr\'e $(-nK).(-K)= 2n$
obtenu par intersection avec l'image r\'eciproque d'une droite de $\P^2_{k}$.

Notons  $(n+1)^2 -(n+1)+1 = n^2+n+1$.
Soit $n\geq 1$ l'unique entier tel que
$$g=n^2-n+1 \leq d < n^2+n+1.$$

Supposons d'abord 
$$g= n^2-n+1 < d  \leq n^2+n-2.$$

Comme on a
$d+2 \leq n^2+n$,
le th\'eor\`eme \ref{fertile}(a) garantit l'existence d'une courbe $\Gamma$
projective, lisse, g\'eom\'etriquement connexe dans le syst\`eme lin\'eaire 
$\lvert -nK \rvert$, 
poss\'edant  un z\'ero-cycle effectif
 de degr\'e $d$
et   un point de degr\'e 2.  
Comme $d$ est impair, on n'a pas $d=n^2-n+2$, donc on a $n^2-n+3 \leq d$
et $d-2 \geq g$. Le th\'eor\`eme de Riemann-Roch sur la courbe $\Gamma$
assure l'existence d'un z\'ero-cycle effectif  de degr\'e $d-2$ sur $\Gamma$, donc sur $X$, 
ce qui est une contradiction avec l'hypoth\`ese $d$ minimal.

On ne  peut avoir $d=n^2+n$, car $d$ est impair.
Il reste donc \`a consid\'erer les cas $d=n^2+n-1$ 
et $d=n^2-n+1=g$.
  
 Consid\'erons le cas $d=n^2+n-1$. Le th\'eor\`eme  
\ref{fertile}(a) \'etablit l'existence d'une courbe $\Gamma$ projective, lisse, g\'eom\'etriquement connexe
sur $k$ de genre $g=n^2-n+1$ contenant un 
z\'ero-cycle effectif
 de degr\'e $d=n^2+n-1$.
Comme on a remarqu\'e ci-dessus, cette courbe contient aussi un z\'ero-cycle de degr\'e $2n$.
Comme $n^2+n-1$ et $2n$ sont premiers entre eux, cette courbe poss\`ede un z\'ero-cycle de degr\'e 1.
Par le th\'eor\`eme de Riemann-Roch sur la courbe  $\Gamma$, 
elle poss\`ede un z\'ero-cycle effectif de degr\'e $n^2-n+1$.
On a $n^2-n+1 < n^2+n-1$ si et seulement si $n>1$, i.e. $d \geq 5$.  Si donc {\it $d$
n'est pas \'egal \`a $3$}, on trouve sur $\Gamma$ et donc sur $X$ 
un z\'ero-cycle effectif de degr\'e impair plus petit que $d$,   ce qui est une contradiction avec
l'hypoth\`ese $d$ minimal.

Il reste \`a consid\'erer le cas $g=n^2-n+1 = d$. 
On a un point ferm\'e de degr\'e $d\geq 3$. Comme $k$ est fertile,
on peut choisir un point $k$-rationnel g\'en\'eral  $m$ dans $\P^2_{k}$
et son image r\'eciproque $M$ par le morphisme $X \to \P^2_{k}$.
C'est un point ferm\'e de degr\'e 2. 
On consid\`ere $p: Y \to X$ l'\'eclat\'e de $X$ en le point $M$,
on note $E \subset Y$ le diviseur exceptionnel, et on consid\`ere
sur $Y$ le syst\`eme lin\'eaire 
$\lvert p^*(-nK)-2E \rvert$. 
Ses sections correspondent aux courbes du syst\`eme lin\'eaire  
 $\lvert -nK \rvert$ 
sur $X$ qui ont un point double en le point ferm\'e $M$,
ce qui impose 6 conditions lin\'eaires.
On a donc
 $h^0(Y, p^*(-nK)-2E) \geq  n^2+n-5$.
 
 \begin{lemma}\label{calculRR2}
 Pour $n \geq 3$, le faisceau inversible
 $p^*(-nK)-2E$ est
 tr\`es ample, et l'on a  $h^0(Y, p^*(-nK)-2E) =  n^2+n-5$.
 \end{lemma}
 \begin{proof} Soit $D\simeq \P^1_{k}$ la droite param\'etrant les droites de $\P^2_{k}$
 passant par $m$. 
  \`A tout point de $X$ non au-dessus de $m$ on associe sa projection
dans $\P^2_{k}$ puis le point de $D$ correspondant \`a la droite
joignant cette projection \`a $m$. L'application rationnelle de $X$
vers $D$ ainsi d\'efinie s'\'etend en un morphisme $Y \to D$
dont le syst\`eme lin\'eaire associ\'e est donn\'e par le faisceau inversible $p^*(-K)-E$.
On sait que le faisceau inversible $-2K$ sur $X$ est tr\`es ample,
d\'efinissant un plongement $X \subset \P^N$.
On a un  plongement 
$$Y  \hookrightarrow (\P^1 \times X) \hookrightarrow ( \P^1 \times \P^N )$$
d\'efini  par  
 $p^*(-K)-E$ pour la projection vers  $\P^1$ et par   $p^*(-2K) $  pour la projection vers $X \subset \P^N$.
 Il s'en suit que pour tout couple d'entiers $a \geq 1, b\geq 1$
  le faisceau inversible $a(p^*(-K)-E)  + b p^*(-2K)$ est tr\`es ample.
  
  En particulier, pour $n=4$, le faisceau  inversible $p^*(-nK)-2E$ est tr\`es ample sur $Y$,
  et comme $p^*(-K)$ correspond \`a un morphisme $Y \to X \to \P^2_{k}$, 
  ceci implique que pour tout $n \geq 4$, le faisceau inversible 
  $p^*(-nK)-2E$ est tr\`es ample sur $Y$.
  
  Proc\'edant comme dans le lemme \ref{calculRR3}, pour $n \geq 4$, on montre  $$h^1(Y, p^*(-nK)-2E)=0, \
 h^1(Y, p^*(-nK)-2E)=0,$$ 
  puis $h^0(Y, p^*(-nK)-2E) =  n^2+n-5$.
 \end{proof}
 
Si l'on a $d\leq n^2+n-6$, c'est-\`a-dire $n^2-n+1 \leq n^2+n-6$,
c'est-\`a-dire $n >3$, c'est-\`a-dire si {\it on exclut   $d=3$ et $d=7$},
le th\'eor\`eme \ref{fertile}(a) assure l'existence
d'une courbe $\Gamma$ g\'eom\'etriquement connexe et lisse dans le syst\`eme lin\'eaire 
$\lvert p^*(-nK)-2E \rvert$ 
 sur $Y$
  contenant un z\'ero-cycle effectif
 de degr\'e $d=n^2-n+1$. 
Cette courbe satisfait  
$g(\Gamma)=p_{a}(\Gamma)= n^2-n-1$. Elle contient
  un z\'ero-cycle de degr\'e $2n$. Comme  
$d=n^2-n+1$ et $2n$ sont premiers entre eux, $\Gamma$ contient un z\'ero-cycle de degr\'e 1.
Par le th\'eor\`eme de Riemann-Roch sur $\Gamma$, elle contient un z\'ero-cycle effectif
de degr\'e $n^2-n-1$, impair et strictement plus petit que $d=n^2-n+1$, 
 ce qui est une contradiction avec
l'hypoth\`ese $d$ minimal.

N'ont donc \'et\'e exclus de ce processus de descente des degr\'es impairs que les  
 degr\'es $1$, $3$, ou $7$.
 \end{proof}

\begin{rmk}\label{contrexKM}
Comme annonc\'e dans  \cite[Remark19]{KM}, pour une surface de del Pezzo de degr\'e 2,
on ne peut exclure la possibilit\'e d'existence d'un point de degr\'e~3 en l'absence de point rationnel.
Je d\'etaille ici  l'argument qui m'a \'et\'e indiqu\'e par J.~Koll\'ar.
Sur un corps $k$ convenable de caract\'eristique z\'ero, on peut trouver dans $\P^2_{k}$
une conique lisse $C(u,v,w)=0$ et une quartique lisse $Q(u,v,w)=0$ dont l'intersection consiste en la r\'eunion
d'un point ferm\'e  de corps r\'esiduel $K$ degr\'e 3 sur $k$ et d'un point ferm\'e  de corps r\'esiduel $L$ de degr\'e 5 sur $k$.
 En particulier cette intersection
ne contient pas de point rationnel.

Soit $F=k(t)$ le corps des fonctions rationnelles en une variable. La quartique de $\P^2_{F}$ d\'efinie par $aC(u,v,w)^2-tQ(u,v,w)=0$ est lisse, car elle se
sp\'ecialise en $t=\infty$  en une quartique lisse. On consid\`ere la surface de del Pezzo $X$ de degr\'e 2 sur $F$ d\'efinie par l'\'equation multihomog\`ene $$z^2-aC(u,v,w)^2+tQ(u,v,w)=0.$$
Supposons qu'elle ait un point sur $F$. Par  congruences modulo $t$, 
on voit que l'on devrait avoir une solution non triviale pour $C(u,v,w)=0=Q(u,v,w)$ dans $k$,
ce qui n'est pas.  Ainsi $X(F) = \emptyset$. Il est par contre clair que $X$
poss\`ede un point sur l'extension cubique $K(t)/F$ et un  point sur l'extension quintique $L(t)/F$,
 avec $C(u,v,w)=0=Q(u,v,w)$ et $z=0$.
 
On peut aussi faire des variantes avec $F=k((t))$ le corps des s\'eries formelles.
 Dans la situation parall\`ele des surfaces fibr\'ees en coniques sur $\P^1_{F}$
avec 6 fibres g\'eom\'etriques d\'eg\'en\'er\'ees, des exemples analogues avec $F$
un corps $p$-adique avaient \'et\'e construits dans \cite[\S 5]{CTC}.

\end{rmk}

\subsection{Surfaces de del Pezzo de degr\'e 2 avec un point rationnel}\label{dp2pointrat}

\begin{thm}\label{soustractionpointdp2}
Soient $k$ un corps de caract\'eristique z\'ero et $X$ une surface de del Pezzo de degr\'e 2 sur $k$ poss\'edant un
point rationnel.
 
(a) Soit $Q\in X(k)$ un point rationnel.
Tout z\'ero-cycle effectif de degr\'e au moins 6 sur $X$
est rationnellement  \'equivalent \`a un z\'ero-cycle
effectif $z_{1}+rQ$ avec $r \geq 0$ et $z_{1}$
effectif de degr\'e au plus 6.

(b) Tout z\'ero-cycle de degr\'e positif ou nul est rationnellement \'equivalent
\`a une diff\'erence $z_{1}-z_{2}$ avec $z_{1}$  effectif
et  $z_{2}$ efffectif de degr\'e au plus 6. 

(c) Tout z\'ero-cycle de degr\'e z\'ero
est  rationnellement \'equivalent \`a la diff\'erence de deux cycles effectifs de degr\'e 6. 

(d) Tout z\'ero-cycle de degr\'e au moins \'egal \`a 43 est rationnellement
 \'equivalent \`a un z\'ero-cycle effectif.
\end{thm}

\begin{proof}
On va syst\'ematiquement appliquer le th\'eor\`eme \ref{fertile}(b).
\`A la diff\'e\-rence du cas des surfaces cubiques lisses (Th\'eor\`eme \ref{soustractionpointdp3}),
 en pr\'esence
d'un $k$-point sur la surface de del Pezzo de degr\'e 2, 
la $k$-unirationalit\'e et la propri\'et\'e de densit\'e ne sont pas
connues dans tous les cas \cite{Ma,STVA}. En outre, pour ces surfaces,
on n'a pas \'etudi\'e la propri\'et\'e de  $R$-densit\'e.
 On va donc utiliser ici la proposition \ref{reducfertile},
  qui permet de supposer le corps $k$ fertile,
 et  le th\'eor\`eme  \ref{ratconnexeRdense}.

  Soit $n \geq 1$.  On a $h^0(-nK)=n^2+n+1$, et si $\Gamma$ est une courbe g\'eom\'etriquement connexe lisse
   dans le syst\`eme lin\'eaire $\vert -nK \vert$, alors $g(\Gamma)=p_{a}(\Gamma)= n^2-n+1$.
 Pour tout $n \geq 1$, le syst\`eme lin\'eaire $\vert -nK\vert$ d\'efinit un morphisme de $X$
 dans un espace projectif d'image de dimension au moins 2. Pour $n\geq 2$, c'est un plongement.
 
 Soit $z$ un z\'ero-cycle effectif  de degr\'e $d\geq 1$.
Soit $n$ le plus petit entier tel que
$d+2 \leq n^2+n+1$. On a   $ n^2-n \leq  d.$

D'apr\`es le th\'eor\`eme \ref{fertile}(b),
quitte \`a remplacer le z\'ero-cycle effectif $z$ par un z\'ero-cycle effectif rationnellement \'equivalent
encore not\'e $z$
et $Q$ par un point rationnel rationnellement \'equivalent encore not\'e $Q$,  
comme on a $h^0(-nK) \geq d+2$,
on peut supposer qu'il existe
 une courbe lisse  g\'eom\'etriquement connexe $\Gamma$ dans le syst\`eme lin\'eaire 
$\lvert -nK \rvert$
contenant  le  z\'ero-cycle  $z$  et  le point rationnel $Q$.

Si l'on a $n^2-n+1 \leq d-1$, alors   le z\'ero-cycle $z-Q$ est rationnellement \'equivalent sur $\Gamma$, donc sur $X$, \`a un z\'ero-cycle effectif de degr\'e $d-1$.

La condition est satisfaite sauf si
$$n^2-n \leq d \leq n^2+n+1.$$

\medskip

Consid\'erons le cas $d =n^2-n+1$. 
 Ici $p_{a}(\Gamma)=d.$
Dans ce cas, on fixe un autre point rationnel $R \in X(k)$, distinct de $Q$,
 non dans le support de $z$, et  situ\'e ni sur une des courbes exceptionnelles de $X$
 ni sur le lieu de ramification du rev\^{e}tement double $X \to \P^2_{k}$ d\'efini par
 le syst\`eme lin\'eaire $\vert -K \vert$.
 Quitte \`a remplacer par des cycles effectifs rationnellement \'equivalents,
on cherche une courbe $\Gamma$ g\'eom\'etriquement int\`egre dans le syst\`eme lin\'eaire
$\vert -nK \vert$ contenant le point $Q$, le support de $z$, et poss\'edant un $k$-point
double en $R$,  
pour faire baisser le genre g\'eom\'etrique de 1.
On veut donc
$$ d+1+3 +1 \leq n^2+n+1,$$
avec $d =n^2-n+1$, 
soit  $n > 2 $ et  $d >3$.

Voici comment faire cela pr\'ecis\'ement.
On consid\`ere l'\'eclatement $p : Y \to X$ en le point $R$, 
 la courbe exceptionnelle
 $E \subset Y$, et le faisceau inversible
 $p^*(-nK)-2E$ sur $Y$.
 La surface $Y$ est une surface de del Pezzo de degr\'e 1. 
D'apr\`es le lemme \ref{calculRR2},
pour tout couple d'entiers $a \geq 1, b\geq 1$
  le faisceau inversible $a(p^*(-K)-E)  + b p^*(-2K)$ est tr\`es ample sur $Y$.
  Ainsi pour tout $n\geq 3$, le faisceau inversible $p^*(-nK)-2E$ 
  est tr\`es ample sur $Y$. On applique ensuite 
  le th\'eor\`eme \ref{fertile}(b) \`a $Y$,
  au plongement de $Y$ d\'efini par $p^*(-nK)-2E$
  au point $Q_{1}=p^{-1}(Q)$ et au z\'ero-cycle effectif $z_{1}=p^*(z)$.
  On trouve ainsi une courbe $\Gamma_{1} \subset Y$ g\'eom\'etriquement connexe, lisse,
  et contenant un $k$-point $Q_{2}$ rationnellement \'equivalent \`a $Q_{1}$ sur $Y$
  et un z\'ero-cycle effectif $z_{2}$ rationnellement \'equivalent \`a $z_{1}$ sur $Y$.
En utilisant le th\'eor\`eme de Riemann-Roch sur $Y$, on montre  $p_{a}(\Gamma_{1})=n^2-n=d-1$.
Par Riemann-Roch sur la courbe $\Gamma_{1}$, on trouve un z\'ero-cycle effectif $z_{3}$
rationnellement \'equivalent sur $\Gamma_{1}$  \`a $z_{2}-Q_{2}$, donc rationnellement \'equivalent
\`a $z_{1}-Q_{1}$ sur $Y$. Alors le z\'ero-cycle effectif $p_{*}(z_{3})$ de degr\'e $d-1$
est rationnellement \'equivalent \`a $z-Q$ sur $X$.

\medskip

Consid\'erons le cas $d =n^2-n$. 
On a $p_{a}(\Gamma)=d+1.$
  Dans ce cas, choisissons un couple de $k$-points \'etrangers au support de $z$,
  \`a $Q$, aux courbes exceptionnelles de premi\`ere esp\`ece sur $X$
  et au lieu de ramification.
 
  Quitte \`a remplacer  $Q$ et $z$ par des cycles effectifs rationnellement \'equivalents,
on cherche une courbe $\Gamma$ g\'eom\'etriquement int\`egre dans le syst\`eme lin\'eaire
$\vert -nK \vert$ contenant le point $Q$, le support de $z$, et sur laquelle les points
$R$ et $S$ sont doubles, 
afin de faire baisser le genre g\'eom\'etrique de 2. Il faut pour cela
$$d+1+6+1 \leq n^2+n+1,$$
avec $d =n^2-n$. 
On doit donc avoir   $n>3$ et $d >6$.

Voici comment faire cela pr\'ecis\'ement. Choisissons le couple $R,S$ 
stable par 
   l'involution associ\'ee au rev\^etement double $X \to \P^2_{k}$
  d\'efini par $\vert -K \vert$.

On consid\`ere l'\'eclatement $p : Y \to X$  en ces points $R$ et $S$,
   les courbes exceptionnelles
 $E_{R}, E_{S} \subset Y$ introduites par l'\'eclatement, et le faisceau inversible
 $p^*(-nK)-2E_{R}-2E_{S}$ sur $Y$.
 D'apr\`es le lemme \ref{calculRR2},   pour $n \geq 3$, ce faisceau inversible est tr\`es ample sur $Y$.
 
  On applique ensuite 
  le th\'eor\`eme \ref{fertile}(b) \`a $Y$,
  au plongement de $Y$ d\'efini par $p^*(-nK)-2E_{R}-2E_{S}$
  au point $Q_{1}=p^{-1}(Q)$ et au z\'ero-cycle effectif $z_{1}=p^*(z)$.
  On trouve ainsi une courbe $\Gamma_{1} \subset Y$ g\'eom\'etriquement connexe, lisse,
  et contenant un $k$-point $Q_{2}$ rationnellement \'equivalent \`a $Q_{1}$ sur $Y$
  et un z\'ero-cycle effectif $z_{2}$ rationnellement \'equivalent \`a $z_{1}$ sur $Y$.
En utilisant le th\'eor\`eme de Riemann-Roch sur $Y$, on montre  $p_{a}(\Gamma_{1})=n^2-n-1=d-1$.
Par Riemann-Roch sur la courbe $\Gamma_{1}$, on trouve un z\'ero-cycle effectif $z_{3}$
rationnellement \'equivalent sur $\Gamma_{1}$  \`a $z_{2}-Q_{2}$, donc rationnellement \'equivalent
\`a $z_{1}-Q_{1}$ sur $Y$. Alors le z\'ero-cycle effectif $p_{*}(z_{3})$ de degr\'e $d-1$
est rationnellement \'equivalent \`a $z-Q$ sur $X$.

Ceci \'etablit (a). Les \'enonc\'es (b) et (c) sont des cons\'equences imm\'ediates.

\medskip

Montrons (d).  
Soit $z$ un z\'ero-cycle  quelconque de degr\'e $d\geq 0$. 
D'apr\`es (b), il est rationnellement
\'equivalent \`a $z_{1} - z_{2} $ avec $z_{1}$ effectif et $z_{2}$ effectif de degr\'e 6.

Le plus petit entier $d$ pour lequel on a
$n^2-n+1  \leq d-6$ et $d+6+1 \leq n^2+n+1$ est $d=49$,
avec $n=7$.
On consid\`ere d'abord le cas o\`u le z\'ero-cycle effectif $z_{1}$ est degr\'e 
$d=49$.
On utilise l'hypoth\`ese $k$ fertile et le th\'eor\`eme \ref{fertile}(b).
Quitte \`a remplacer les z\'ero-cycles effectifs  $z_{1} $ et $z_{2}$ 
par des z\'ero-cycles effectifs rationnellement \'equivalents,
dans le syst\`eme lin\'eaire $\vert -7K \vert$ qui v\'erifie $h^0(-7K)=n^2+n+1=57 > 49+6+1$ 
 on trouve   une courbe $\Gamma$ g\'eom\'etriquement irr\'eductible
et lisse de genre $n^2-n+1= 43$ 
 qui contient
les supports de $z_{1}$ et $z_{2}$. Le z\'ero-cycle $z_{1}-z_{2}$
de degr\'e $43$ est rationnellement \'equivalent sur $\Gamma$, donc sur $X$
\`a un z\'ero-cycle effectif.
 
Ceci implique que tout z\'ero-cycle  $z_{1} - z_{2} $ sur $X$ 
  avec $z_{1}$ effectif  
de degr\'e $d\geq  49$  et $z_{2}$ effectif de degr\'e 6 est rationnellement \'equivalent \`a un z\'ero-cycle effectif.

Ainsi tout z\'ero-cycle $z$  sur $X$ de degr\'e au moins \'egal \`a 43 est rationnellement \'equivalent \`a un z\'ero-cycle
effectif.
 \end{proof}

   \begin{rmk}
   La d\'emonstration \'etablit   que pour  $Q \in X(k)$ donn\'e,
   tout z\'ero-cycle effectif  de degr\'e $d\geq 1$ sur $X$ est rationnellement
   \'equivalent \`a $z_{1}  +Q$ avec $z_{1}$ z\'ero-cycle effectif, si 
   $d \notin \{1,2,3,6\}$.
    \end{rmk}

\section{Surfaces de del Pezzo de degr\'e 1}\label{dp1pointrat}

\begin{thm}\label{soustractionpointdp1}
Soit $X/k$ une surface de del Pezzo de degr\'e 1.
 
(a) Soit $Q\in X(k)$ un point rationnel.
Tout z\'ero-cycle effectif de degr\'e au moins 21 sur $X$
est rationnellement  \'equivalent \`a un z\'ero-cycle
effectif $z_{1}+rQ$ avec $r \geq 0$ et $z_{1}$
effectif de degr\'e au plus 21.

(b) Tout z\'ero-cycle de degr\'e positif ou nul est rationnellement \'equivalent
\`a une diff\'erence $z_{1}-z_{2}$ avec $z_{1}$  effectif
et  $z_{2}$ effectif de degr\'e au plus 21. 

(c) Tout z\'ero-cycle de degr\'e z\'ero
est  rationnellement \'equivalent \`a la diff\'erence de deux cycles effectifs de degr\'e 21.

(d)  Tout z\'ero-cycle sur $X$ de degr\'e au moins \'egal \`a $904$  est rationnellement
\'equivalent \`a un z\'ero-cycle effectif.
\end{thm}
\begin{proof}
On va syst\'ematiquement appliquer le th\'eor\`eme \ref{fertile}(b). 
On dispose ici automatiquement d'un
$k$-point $P$, le point fixe du syst\`eme lin\'eaire $\vert -K \vert$, qui satisfait $h^0(-K)=2$.
Pour tout $n \geq 2$, le syst\`eme lin\'eaire $\vert -nK \vert$ 
est sans point base  \cite[Chap. III, Prop. 3.4]{Kol},  et son image est de dimension 2.
Pour tout $n \geq 3$, le faisceau inversible $-nK$ est tr\`es ample.

On ne conna\^{\i}t en g\'en\'eral pas la $k$-unirationalit\'e,
la propri\'et\'e de densit\'e, et encore moins   la propri\'et\'e de  $R$-densit\'e.
 On va donc utiliser la proposition \ref{reducfertile},
  qui permet de supposer le corps $k$ fertile,
 et  le th\'eor\`eme \ref{ratconnexeRdense}.

  Soit $n \geq 1$.  On a $h^0(-nK)=n(n+1)/2+1$, 
  et si $\Gamma$ est une courbe g\'eom\'e\-tri\-quement connexe lisse
   dans le sys\`eme lin\'eaire $\vert -nK \vert$, alors $g(\Gamma)=p_{a}(\Gamma)= n(n-1)/2+1$.
 
 Soit $z$ un z\'ero-cycle effectif  de degr\'e $d\geq 2$.
Soit $n$ le plus petit entier tel que
$d+2 \leq n(n+1)/2+1$. 
On a   $ n(n-1)/2  \leq  d$ et  $n \geq 2$.

D'apr\`es le th\'eor\`eme \ref{fertile}(b), sous l'hypoth\`ese  $d+2 \leq n(n+1)/2+1$ et $n\geq 3$,
quitte \`a remplacer le z\'ero-cycle effectif $z$ par un z\'ero-cycle effectif rationnellement \'equivalent
encore not\'e $z$, \'etranger \`a $P$
et $Q$ par un point rationnel rationnellement \'equivalent encore not\'e $Q$, 
\'etranger aux pr\'ec\'edents,
on peut supposer qu'il existe
 une courbe lisse  g\'eom\'etriquement connexe $\Gamma$ dans le syst\`eme lin\'eaire 
$\lvert -nK \rvert$
contenant  le  z\'ero-cycle  $z$  et  le point rationnel $Q$.

Si l'on a $n(n-1)/2+1 \leq d-1$,
 alors le z\'ero-cycle $z-Q$ est rationnellement \'equivalent sur $\Gamma$, 
 donc sur $X$, \`a un z\'ero-cycle effectif de degr\'e $d-1$.

La condition est satisfaite sauf si
$$n(n-1)/2 \leq d \leq n(n-1)/2 +1.$$

\medskip

Consid\'erons le cas $d =n(n-1)/2 +1$. 
 Ici $p_{a}(\Gamma)=d.$
 Quitte \`a remplacer $z$ et $Q$  par des cycles effectifs rationnellement \'equivalents,
on cherche une courbe $\Gamma$ g\'eom\'etriquement int\`egre dans le syst\`eme lin\'eaire
$\vert -nK \vert$ contenant  $Q$, 
 le support de $z$, et poss\'edant un $k$-point
double en  un point rationnel  $P$ \'etranger aux pr\'ec\'edents,
pour faire baisser le genre g\'eom\'etrique de 1.
On veut donc
$$ d+1+3 +1 \leq n(n-1)/2 +1,$$
avec $d =n(n-1)/2 +1$, soit $n>4$ et $d>7$.
 On consid\`ere l'\'eclatement $p : Y \to X$ en pr\'ecis\'ement le point $P$ point fixe du syst\`eme lin\'eaire anticanonique,
 la courbe exceptionnelle
 $E \subset Y$, et le faisceau inversible
 $p^*(-nK)-2E$ sur $Y$.
 
 Le syst\`eme lin\'eaire $\vert p^*(-K) -E \vert$ sur $Y$ 
 d\'efinit un morphisme $Y \to \P^1_{k}$ correspondant au pinceau 
 de courbes de genre arithm\'etique 1 d\'efinies par $-K$ sur $X$, surface de del Pezzo de degr\'e 1.
 Sur $X$, tout  syst\`eme lin\'eaire $\vert -nK \vert$ avec $n \geq 2$
 d\'efinit un morphisme dans un espace projectif, d'image de dimension 2.
 On conclut que sur $Y$,  les sections de tout faisceau inversible de la forme $p^*(-nK) + m(p^*(-K) -E)$
 avec $n \geq 2$ et $m\geq 1$ d\'efinissent un morphisme de $Y$ dans un espace projectif
 d'image de dimension 2. Ainsi, pour $n \geq 3$, les sections du faisceau inversible  $p^*(-nK)-2E$
 d\'efinissent  un morphisme de $Y$ dans un espace projectif
 d'image de dimension 2.  
 
 Comme le groupe des sections de $-nK$ sur $X$ ayant un point double en $P$ s'injecte dans le
 groupe des sections de $p^*(-nK)-2E$ sur $Y$, on a, sous l'hypoth\`ese $n>4$ ou encore $d>7$, 
 $$h^0(Y, p^*(-nK)-2E) \geq  n(n-1)/2 +1 -3 \geq d+2.$$
 Par ailleurs le genre de toute courbe g\'eom\'etriquement connexe lisse
 dans le syst\`eme lin\'eaire $\vert p^*(-nK)-2E) \vert$ est \'egal \`a $(n^2-n)/2=d-1$.

 D'apr\`es le th\'eor\`eme \ref{fertile}(b), sous l'hypoth\`ese $n\geq 3$,
quitte \`a remplacer le z\'ero-cycle effectif $p^*(z)$ par un z\'ero-cycle effectif rationnellement \'equivalent
  $z_{1}$, \'etranger \`a $p^*(Q)$
et $p^*(Q)$ par un point rationnel  $Q_{1} \in Y(k)$ rationnellement \'equivalent, 
il existe
 une courbe lisse  g\'eom\'etriquement connexe $\Gamma$ dans le syst\`eme lin\'eaire 
$\lvert p^*(-nK)-2E \rvert$
contenant  le  z\'ero-cycle  $z_{1}$  et  le point rationnel $Q_{1}$. Cette courbe est de genre $d-1$.
On trouve donc sur elle un z\'ero-cycle effectif de degr\'e $d-1$ rationnellement \'equivalent
\`a $z_{1}-Q_{1}$. L'image directe sur $X$ donne un z\'ero-cycle effectif de degr\'e $d-1$
rationnellement \'equivalent \`a $z-Q$.

\medskip

Consid\'erons le cas $d =n(n-1)/2$. 
  Ici $p_{a}(\Gamma)=d+1.$
  Dans ce cas, on  va fixer deux  autres points rationnels $R, S \in X(k)$, distincts de $Q$.
  Quitte \`a remplacer par des cycles effectifs rationnellement \'equivalents,
on cherche une courbe $\Gamma$ g\'eom\'etriquement int\`egre dans le syst\`eme lin\'eaire
$\vert -nK \vert$ contenant les point $Q$, le support de $z$, et sur laquelle les points
$R$ et $S$ sont doubles, 
pour faire baisser le genre g\'eom\'etrique de 2. Il faut pour cela
$$d+1+6+1 \leq n(n+1)/2+1,$$
avec $d =n(n-1)/2$.
On doit donc avoir $n>6$ et  $d>15$.

On consid\`ere l'\'eclatement $p : Y \to X$  en les points $R$ et $S$,
   les courbes exceptionnelles
 $E_{R}, E_{S} \subset Y$, et le faisceau inversible
 $p^*(-nK)-2E_{R}-2E_{S}$ sur $Y$.

  \begin{lemma}\label{lem3}
  Soient $k$ un corps  et $X$ une $k$-vari\'et\'e projective.
   Soit $f : X \to \P^n$ un $k$-morphisme, $Z \subset \P^n_{k}$ son image sch\'ematique.
  Soit $Y \to X$ l'\'eclat\'e de $Y$ en un $k$-point lisse $m \in X(k)$, et soit
  $E \subset Y$ le diviseur exceptionnel.
  Si $f : X \to Z $ est \'etale dans un voisinage de $m$,
 alors le syst\`eme lin\'eaire $\vert f^*(O_{\P^n}(1)) \otimes O_{Y}(-E) \vert$ sur $Y$  est
 sans point base.
 \end{lemma}
   \begin{proof} Pour \'etablir cela, on peut supposer $k$ alg\'ebriquement clos.
   Il suffit alors d'utiliser le fait qu'au point $m$ le morphisme $f$ s\'epare les points infiniment voisins.
   \end{proof}

   Sur la surface de del Pezzo $X$ de degr\'e 1,
le syst\`eme lin\'eaire $\vert -2K \vert$
   d\'efinit  un morphisme $f: X \to   \P^3_{k}$ d'image de dimension 2.
   Comme on a suppos\'e ${\rm car}(k)=0$, ce  morphisme est g\'en\'eriquement \'etale.
   Soit $S \in X(k)$ un point o\`u $f$ est \'etale. Soit $g_{S}: X_{S} \to X$ l'\'eclatement au point $S$ et $E_{S} \subset X_{S}$
   la courbe exceptionnelle.
   D'apr\`es le lemme \ref{lem3}, le syst\`eme lin\'eaire $\vert g_{S}^*(-2K)-E_{S} \vert$ est sans point base et d\'efinit un morphisme surjectif $X_{S} \to \P^2_{k}$.
   Tout multiple de $g_{S}^*(-2K)-E_{S} \in {\rm Pic}(X_{S})$ est sans point base et d\'efinit un morphisme d'image de
   dimension 2.
   
   Pour le point $R$ annonc\'e plus haut on va choisir le point $P$ qui est le point base
   du syst\`eme lin\'eaire $\vert -K \vert$ sur $X$.
   Soit $g_{P}: X_{P} \to X$ l'\'eclat\'e de $X$ en $P$  
   et $E_{P} \subset X_{P}$
   la courbe exceptionnelle.

   On a vu ci-dessus que  $g_{P}^*(-K) -E_{P}$ d\'efinit un morphisme $X_{P} \to \P^1$.
Tout multiple de $g_{P}^*(-K) -E_{P}$ d\'efinit donc un morphisme.
 
Rappelons que  pour $n \geq 2$, le syst\`eme lin\'eaire $\vert -nK \vert$ 
sur $X$ est sans point base.

Soit $Y$  le produit fibr\'e $X_{P}$ et $X_{S}$ au-dessus de $X$.
 C'est l'\'eclat\'e   $p: Y \to X$ de $X$ en $P$ et $S$.
On note encore $E_{P}$ et $E_{S}$ les diviseurs exceptionnels dans $Y$.
Toute combinaison lin\'eaire \`a coefficients entiers $(a,b,c)$
$$  a(p^*(-2K) - E_{S}) +
b(p^*(-K)-E_{P}) + c(p^*(-K)) $$
 avec $a\geq 0, b\geq 0$ et $c=0$ ou $ c \geq 2$  d\'efinit un syst\`eme
lin\'eaire sans point base sur $Y$,    d'image de dimension 2.
Ainsi pour $n=6$ et pour $n \geq 8$,  le syst\`eme lin\'eaire
$\vert p^*(-nK)  -2E_{S}-2E_{P} \vert$ d\'efinit un syst\`eme
lin\'eaire sans point base sur $Y$,  d'image de dimension 2.

D'apr\`es le th\'eor\`eme \ref{fertile}(b), sous l'hypoth\`ese $n\geq 8$,
quitte \`a remplacer le z\'ero-cycle effectif $p^*(z)$ par un z\'ero-cycle effectif rationnellement \'equivalent
  $z_{1}$, \'etranger \`a $p^*(Q)$
et $p^*(Q)$ par un point rationnel  $Q_{1} \in Y(k)$ rationnellement \'equivalent, 
il existe
 une courbe lisse  g\'eom\'etriquement connexe $\Gamma_{1}$ dans le syst\`eme lin\'eaire 
$\lvert p^*(-nK)-2E \rvert$
contenant  le  z\'ero-cycle  $z_{1}$  et  le point rationnel $Q_{1}$. Cette courbe est de genre $d-1$.
On trouve donc sur elle un z\'ero-cycle effectif de degr\'e $d-1$ rationnellement \'equivalent
\`a $z_{1}-Q_{1}$. L'image directe sur $X$ donne un z\'ero-cycle effectif de degr\'e $d-1$
rationnellement \'equivalent \`a $z-Q$.

La condition $n >7$ \'equivaut ici \`a $d = n(n-1)/2 > 21$.

Ceci \'etablit (a). Les \'enonc\'es (b) et (c) sont des cons\'equences imm\'ediates.

\medskip

Montrons (d).  
Soit $z$ un z\'ero-cycle  quelconque de degr\'e $d\geq 0$. 
D'apr\`es (b), il est rationnellement
\'equivalent \`a $z_{1} - z_{2} $ avec $z_{1}$ effectif et $z_{2}$ effectif de degr\'e 21.

Le plus petit entier $d$ pour lequel on a
$n(n-1)/2+1  \leq d-21$ et $d+21+1 \leq n(n+1)/2 +1$ est $d=925$,
avec $n=43$.

On consid\`ere d'abord le cas o\`u le z\'ero-cycle effectif $z_{1}$ est degr\'e 
$d=925$.
On utilise l'hypoth\`ese $k$ fertile et le th\'eor\`eme \ref{fertile}(b).
Quitte \`a remplacer les z\'ero-cycles effectifs  $z_{1} $ et $z_{2}$ 
par des z\'ero-cycles effectifs rationnellement \'equivalents,
dans le syst\`eme lin\'eaire $\vert -43K \vert$ qui v\'erifie $h^0(-43K)=n(n+1)/2 +1= 947$ 
 on trouve   une courbe $\Gamma$ g\'eom\'etriquement irr\'eductible
et lisse de genre $n(n-1)/2+1=904$ 
 qui contient
les supports de $z_{1}$ et $z_{2}$. Le z\'ero-cycle $z_{1}-z_{2}$
de degr\'e $904$ est rationnellement \'equivalent sur $\Gamma$, donc sur $X$
\`a un z\'ero-cycle effectif.
 
Ceci implique que tout z\'ero-cycle  $z_{1} - z_{2} $ sur $X$ 
  avec $z_{1}$ effectif  
de degr\'e $d\geq  925$  et $z_{2}$ effectif de degr\'e $21$  est rationnellement \'equivalent \`a un z\'ero-cycle effectif.

Ainsi tout z\'ero-cycle $z$  sur $X$ de degr\'e au moins \'egal \`a $904$  est rationnellement \'equivalent \`a un z\'ero-cycle
effectif.
\end{proof}

 \begin{rmk}
 La d\'emonstration \'etablit   que pour  $Q \in X(k)$ donn\'e,
   tout z\'ero-cycle effectif  de degr\'e $d\geq 1$ sur $X$ est rationnellement
   \'equivalent \`a $z_{1}  +Q$ avec $z_{1}$ z\'ero-cycle effectif, si  l'on a
  $d \notin \{1, 2, 3,4, 6, 10, 15, 21\}$. Il est tr\`es vraisemblable que l'on 
  pourrait \'eliminer le cas $d=21$. Ce serait le cas 
 si  avec les notations de la d\'emonstration ci-dessus  on savait que
 le syst\`eme lin\'eaire  $\vert p^*(-7K)  -2E_{S}-2E_{P} \vert$  sur $Y$  
 est sans point base. 
 
 Si l'on  pouvait \'eliminer le cas $d=21$, alors au point (d) on pourrait remplacer 904
 par 466.
 \end{rmk}

\section{Surfaces g\'eom\'etriquement rationnelles}\label{total}

Soient $k$ un corps 
et $X$ une $k$-surface projective et lisse
g\'eom\'etriquement rationnelle. Le th\'eor\`eme d'Enriques-Manin-Iskovskikh \cite{I}
et Mori dit qu'une telle surface $k$-minimale
 est $k$-isomorphe soit \`a une
surface projective et lisse fibr\'ee en coniques  (g\'en\'eriquement lisses) relativement minimale
au-dessus d'une conique lisse, soit \`a une surface (lisse) de del Pezzo.
 Une surface de del Pezzo $X$  est une surface dont le faisceau anticanonique $-K$
est ample.

Pour une surface fibr\'ee en coniques relativement minimale $X \to C$ au-dessus
d'une conique lisse, les fibres g\'en\'erales sont des coniques lisses.
Les fibres singuli\`eres $X_{P}$ sont form\'ees de deux droites conjugu\'ees
se rencontrant transversalement au-dessus d'un point ferm\'e s\'eparable $P$.

 \begin{thm}\label{pluspetitpremier}
 Soient $k$ un corps de caract\'eristique z\'ero et
 $X$ une $k$-surface projective, lisse, g\'eom\'etriquement rationnelle.
Il existe un entier $N(X)$, 
 qui ne d\'epend que de
 la g\'eom\'etrie de $X$ sur une cl\^{o}ture alg\'ebrique de $k$,  tel que
si  $X$ poss\`ede un z\'ero-cycle de degr\'e 1, alors $X$
poss\`ede des points ferm\'es dont les  degr\'es sont  inf\'erieurs ou \'egaux \`a $N(X)$
et  sont premiers entre eux  dans leur ensemble.
 Notant  $K_{X}$
 la classe canonique de $X$, on peut prendre
 $$N(X) =  {\rm max}(10, \lfloor{4- (K_{X}.K_{X})/2) }\rfloor).$$

  \end{thm}
 
  \begin{proof}

 Consid\'erons d'abord le cas d'une surface de del Pezzo. Soit
$d=(K.K)$ son degr\'e. Une telle surface poss\`ede un z\'ero-cycle effectif
de degr\'e $d$.
 On a $1 \leq d \leq 9$.
Supposons que $X$ poss\`ede un  z\'ero-cycle de degr\'e 1.
Pour $d \geq 5$, c'est un r\'esultat classique qu'alors $X$ poss\`ede un point
rationnel  : pour $d=5, 7$ il existe toujours un point rationnel \cite{Ma66,VA}.
Pour $d=8$, 
l'existence d'un z\'ero-cycle de degr\'e 1 implique  que toute classe
dans le groupe de Picard g\'eom\'etrique invariante sous l'action du groupe de Galois
de $k$ est dans l'image du groupe de Picard de $X$. La moiti\'e de la classe
anticanonique de $X$ d\'efinit alors plongement de $X$ dans $\P^3_{k}$ dont l'image
est une quadrique. 
On se ram\`ene  ainsi \`a l'\'enonc\'e pour les quadriques de $\P^3_{k}$,
pour lesquelles le r\'esultat est un
 cas particulier du  th\'eor\`eme de   Springer \cite{Spr}  pour les quadriques quelconques. 
 Pour $d=9$, 
$X$ est une surface de Severi--Brauer. Le cas $d=6$ est plus subtil.
On peut l'\'etablir en utilisant le th\'eor\`eme de
 Manin \cite[Chap. IV, Thm. 30.3.1]{Ma} que $X$ contient un ouvert
qui est un espace principal homog\`ene sous un $k$-tore.

Pour $d=4$, l'existence d'un point rationnel fut \'etablie par Coray \cite{C2}. 
On laisse au lecteur le soin de simplifier \cite{C2} suivant la m\'ethode
de cet article.  Pour $d=4$, $X$ est une intersection de deux quadriques
dans $\P^4_{k}$. Par des m\'ethodes \'el\'ementaires,
 M. Amer (non publi\'e)  et  A. Brumer  \cite{B} montr\`erent ensuite que,
pour tout entier naturel $m$,  toute intersection  
de deux quadriques dans $\P^m_{k}$ qui poss\`ede un point
dans une extension de $k$ de degr\'e impair poss\`ede un point rationnel.

Dans tous ces cas, on peut prendre $N(X)=1$.
 
Pour $d=3$, le th\'eor\`eme de Coray  \cite{C1} repris 
au paragraphe \ref{cubique1}
ci-dessus 
donne un point dans une extension
de degr\'e 1, 4 ou 10 et dans une extension de degr\'e 1 ou 3.
On peut prendre $N(X)=10$.

Pour $d=2$, le th\'eor\`eme \ref{DP2} donne un point dans une extension
de degr\'e 1, 3 ou 7 et dans une extension de degr\'e 1 ou 2.
On peut prendre $N(X)=7$.

Toute surface de del Pezzo de degr\'e 1 poss\`ede un point rationnel canonique,
le point fixe du syst\`eme lin\'eaire anticanonique. Ici $N(X)=1$.

\medskip

On v\'erifie facilement que si $Y \to X$ est un $k$-morphisme birationnel
 de $k$-vari\'et\'es projectives, lisses, g\'eom\'etriquement connexes, 
 si $X$ satisfait la propri\'et\'e ci-dessus avec $N(X)$, alors $Y$
 satisfait la propri\'et\'e avec $N(Y)=N(X)$.
 Par ailleurs, si   $Y \to X$ est un $k$-morphisme birationnel de surfaces projectives et lisses,
 qui g\'eom\'etriquement est  obtenu par \'eclatements successifs de $s$ points,
 alors $(K_{Y}.K_{Y})= (K_{X}.K_{X})-s$. 
 Si donc on peut prendre pour $N(X)$
 la fonction  de $K_{X}$  indiqu\'ee \`a la fin du th\'eor\`eme, alors on peut prendre pour
 $N(Y)$ cette fonction de $K_{Y}$.
 
 \medskip
 
 Soit donc d\'esormais $X$ une $k$-surface projective, lisse, g\'eom\'etriquement rationnelle, $k$-minimale,
 poss\'edant un  z\'ero-cycle de degr\'e 1. On a d\'ej\`a \'etabli l'\'enonc\'e avec $N(X)=10$ pour les surfaces de del Pezzo.
Consid\'erons maintenant le cas d'une surface $X$ fibr\'ee en coniques relativement
minimale au-dessus d'une conique $C$ lisse. Comme $X$, la courbe $C$
 poss\`ede un z\'ero-cycle de
degr\'e  1, et donc $C\simeq \P^1_{k}$. La surface $X$ poss\`ede donc 
un point ferm\'e de degr\'e 1 ou 2.
Si on note $r$ le nombre de fibres
g\'eom\'etriques singuli\`eres de la fibration $X\to \P^1_{k}$, on a $r=8 -(K.K)$.
D'apr\`es \cite[Thm. B]{CTC},    si $X$ poss\`ede un point ferm\'e de  degr\'e impair,
alors $X$ poss\`ede un point ferm\'e
 de
  degr\'e impair au plus \'egal \`a ${\rm max}(1, \lfloor{r/2}\rfloor)$,
  valeur que l'on prend pour $N(X)$.
\end{proof}

 \begin{thm}\label{touseffectifs}
  Soit  $X$ une $k$-surface projective, lisse, g\'eom\'e\-tri\-quement rationnelle,
sur un corps $k$ de caract\'eristique z\'ero. Soit $K_{X}$
 la classe canonique de $X$.
Supposons que $X$  poss\`ede un point $k$-rationnel.
Soit $K_{X}$
 la classe canonique de $X$. Il existe un entier $M(X)$, 
 qui ne d\'epend que de
 la g\'eom\'etrie de $X$ sur une cl\^{o}ture alg\'ebrique de $k$,  tel que
tout  z\'ero-cycle de degr\'e au moins $M(X)$ est rationnellement
 \'equivalent \`a un z\'ero-cycle effectif. En particulier, le groupe de Chow des z\'ero-cycles est engendr\'e 
 par les points ferm\'es de degr\'e au plus  $M(X)$.
Notant  $K_{X}$
 la classe canonique de $X$, on peut prendre $$M(X)= {\rm max}(904, \lfloor{3- (K_{X}.K_{X})/2)}\rfloor ).$$
  \end{thm}

\begin{proof}
On v\'erifie que si $X \to Y$ est un $k$-morphisme 
birationnel de $k$-surfaces projectives et lisses g\'eom\'etriquement connexes, et s'il existe
un tel entier $M(Y)$ avec la propri\'et\'e ci-dessus pour $Y$, alors
on a la m\^eme propri\'et\'e pour $X$ avec $M(X)=M(Y)$. Par ailleurs
la fonction de $(K_{X}.K_{X})$ indiqu\'ee dans le th\'eor\`eme est non d\'ecroissante par \'eclatement.
On peut donc supposer la surface $X$ $k$-minimale.

Pour les surfaces  fibr\'ees en coniques au-dessus de $\P^1_{k}$,  relativement minimales
avec $r$ fibres g\'eom\'etriques singuli\`eres,  d'apr\`es \cite[Thm. B]{CTC} on peut
prendre $M(X)={\rm max}(0, \lfloor{r/2}\rfloor -1)$.

On sait que toute $k$-surface de del Pezzo   de degr\'e au moins \'egal \`a 5
avec un point rationnel est $k$-birationnelle \`a un espace projectif \cite{Ma66,VA}.
Dans ce cas, on peut donc prendre $M(X)=0$.
Pour les surfaces de del Pezzo de degr\'e 4, on peut prendre $M(X)=1$ : par \'eclatement d'un $k$-point
non situ\'e sur les 16 droites, on se ram\`ene \`a un fibr\'e en coniques avec $r\leq 5$ fibres g\'eom\'etriques
singuli\`eres, et on peut appliquer le r\'esultat g\'en\'eral ci-dessus.
Pour les  surfaces de del Pezzo de degr\'e 3, le th\'eor\`eme \ref{soustractionpointdp3} donne $M(X)=10$.
Pour les  surfaces de del Pezzo de degr\'e 2,  le th\'eor\`eme \ref{soustractionpointdp2} donne $M(X)=43$.
Pour les surfaces de del Pezzo de degr\'e 1,  le th\'eor\`eme \ref{soustractionpointdp1} donne $M(X)=904$.
\end{proof}

\begin{rmk}
Il resterait \`a \'eliminer l'hypoth\`ese d'existence d'un point rationnel dans le th\'eor\`eme \ref{touseffectifs},
ce qui impliquerait le th\'eor\`eme \ref{pluspetitpremier} avec  l'estimation sans doute trop grossi\`ere $N(X)=M(X)+1$.
\end{rmk}

\section{Surfaces cubiques sans point rationnel}\label{cubique2}

Soit $k$ un corps de caract\'eristique z\'ero.  Soit $X \subset \P^3_{k}$ une surface cubique lisse. Comme rappel\'e plus haut,
si  la surface cubique lisse $X$ poss\`ede un point rationnel, alors  elle est $k$-unirationnelle et
  l'ensemble $X(k)$ de ses points $k$-rationnels est dense dans $X$
 pour la topologie de Zariski.  Il  est donc
facile de trouver 3 points  rationnels sur $X$ qui ne sont pas align\'es dans $\P^3_{k}$.
Le th\'eor\`eme suivant est une r\'eponse partielle \`a la question pos\'ee \`a la fin de l'introduction
du r\'ecent article \cite{QM}.

\begin{thm}\label{aligne}
Soit $X \subset \P^3_{k}$ une surface cubique lisse sur un corps $k$ de caract\'eristique nulle,
 sans point rationnel. Si tout point ferm\'e de degr\'e 3 sur $X$ est d\'ecoup\'e par une
 droite de $\P^3_{k}$, alors \`a toute droite g\'en\'erale de $\P^3_{k}$ on peut associer
  une surface de del Pezzo de degr\'e 1 sur $k$
 dont les points $k$-rationnels ne sont pas denses pour la topologie de Zariski,
 et donc qui en particulier n'est pas $k$-unirationnelle.
\end{thm}

Un point de degr\'e 3 d\'ecoup\'e par une droite sera dit ``align\'e''.

\begin{proof}
Comme $X(k)=\emptyset$, on n'a pas non plus de point quadratique sur $X$.
Soient $\k$ une cl\^oture alg\'ebrique de $k$ et  $G={\rm Gal}(\k/k)$. On note $\X=X \times_{k}\k$.
Soit $D \subset \P^3_{k}$ une droite qui ne rencontre aucune des 27 droites de $\X$.

Le pinceau $L\simeq \P^1_{k}$ des plans de $\P^3_{\k}$ contenant $D$ d\'ecoupe donc sur $\X$ soit
une cubique lisse, soit une cubique avec une unique singularit\'e.
 
Soit $Q=X \cap D$. C'est un point ferm\'e de degr\'e 3,
qui sur $\k$ correspond \`a 3 points dont aucun n'est situ\'e sur une droite de $\X$.
Soit $q= Y \to X$ l'\'eclat\'e de $X$ en $P$.
Soit $R \subset Y$ le diviseur exceptionnel.
Sur $\k$, ceci donne naissance \`a trois courbes $R_{i}$, chacune isomorphe \`a
une droite projective sur $\k$.
La famille des plans passant par $L$ d\'efinit   un
 morphisme $p : Y \to  L$ dont les fibres sont pr\'ecis\'ement les
 cubiques mentionn\'ees ci-dessus. En particulier la fibration
 $Y\times_{k}\k \to \P^1_{\k}$ est relativement minimale.
 
Sur $\k$, le morphisme $p$ admet une section, car le diviseur exceptionnel  $q^{-1}(P)$
se d\'ecoupe en trois courbes isomorphes \`a $\P^1_{\k}$ que $p$
applique isomorphiquement sur $\P^1_{\k}$.

Chaque fibre $Y_{m}=p^{-1}(m)$ au-dessus d'un $k$-point $m \in L(k)$
contient le point ferm\'e $Q$. Supposons $Y_{m}$ lisse. Le th\'eor\`eme de Riemann-Roch sur
la courbe $Y_{m}$, qui est de genre 1, montre qu'un point ferm\'e $Q \in Y_{m}$ qui est 
de degr\'e 3   est  align\'e sur $Y_{m} \subset X$  
si et seulement si le diviseur $Q-P$, qui est de degr\'e z\'ero,
a une classe nulle dans $\Pic(Y_{m})$.

Si donc il existe une classe de degr\'e 3 dans $\Pic(Y_{m})$ qui est distincte de 
la classe de $Q$, alors il existe sur $Y_{m}$, et donc sur $X$, un point ferm\'e
de degr\'e 3 non align\'e.

 Soit $F=k(L)$, resp. $E=\k(L)$  le corps des fonctions rationnelles sur $L$, resp. sur $L_{\k}$.
  Soit $Y_{\eta}/F$ la fibre g\'en\'erique de $p$.
 Soit $W_{\eta}/F$ la jacobienne de la courbe $Y_{\eta}$. C'est une courbe elliptique sur $F$.
 Soit $W \to \P^1_{k}$ le mod\`ele  propre r\'egulier minimal  de  $W_{\eta}/F$ 
 (existence : \cite[Chap. 7]{Sha}; unicit\'e :
 \cite[Chap. 8]{Sha}).

Je dis qu'alors $W_{\k} \to \P^1_{\k}$ est le mod\`ele propre r\'egulier minimal  de 
 la $E$-courbe elliptique $W_{\eta}\times_{F}E$. 
 La minimalit\'e est le point non \'evident. Faute d'avoir trouv\'e une r\'ef\'erence
 dans la litt\'erature, je donne une d\'emonstration.  
Supposons que $W_{\k}$ contienne une courbe exceptionnelle de premi\`ere
esp\`ece $D_{1}$, donc lisse de genre z\'ero et satisfaisant $(D_{1}.D_{1})=-1$,
contenue dans une fibre. Supposons que cette courbe admet
une conjugu\'ee $D_{2}$  sous Galois qui la rencontre, et donc est contenue dans la
m\^{e}me fibre g\'eom\'etrique.  Alors 
$$(D_{1}+D_{2})^2= (D_{1})^2 + (D_{2})^2 + 2 (D_{1}.D_{2}) = -2+ 2 (D_{1}.D_{2}) \geq 0,$$
 donc, vu les propri\'et\'es de la forme d'intersection sur une fibre, qui est semi-d\'efinie n\'egative
 \cite[Chap. 6, p. 91]{Sha}, on a
 $(D_{1}+D_{2})^2=0$, et $D_{1}+D_{2}$ est un multiple rationnel de la fibre contenant 
 $D_{1}$ et $D_{2}$. Cette fibre contient une composante de multiplicit\'e 1,
 comme on voit par intersection avec la section nulle de $W_{\k} \to \P^1_{\k}$.
 Ainsi la fibre est $D_{1}+D_{2}$, et l'on a $(D_{1}.D_{2})=1$. Mais ceci n'est
 pas possible, car le genre  g\'eom\'etrique des fibres serait z\'ero. On voit donc
 que les divers conjugu\'es de $D_{1}$ sont dans des fibres distinctes. Mais alors
 leur somme d\'efinit un diviseur sur la $k$-vari\'et\'e $W$ que  le crit\`ere   de
 Castelnuovo  \cite[Chap. 6, p. 102]{Sha}   permet de contracter, contredisant le fait que $W \to \P^1_k$
 est minimal.

 Comme $Y_{\eta}\times_{F}E$ poss\`ede les points  $E$-rationnels
 correspondant aux courbes $R_{i} $,
 le choix d'une de ces courbes $R_{i}$ d\'efinit un $E$-isomorphisme de courbes :
 $W_{\eta}\times_{F}E \simeq  Y_{\eta}\times_{F}E  $.
 Vu l'unicit\'e des mod\`eles r\'eguliers propres minimaux pour les courbes
lisses de genre au moins 1 \cite[Chap. 8]{Sha}, on voit qu'il existe
  un $\P^1_{\k}$-isomorphisme  $W \times_{k}\k \to  Y \times_{k}\k$ 
induisant l'isomorphisme donn\'e sur les fibres g\'en\'eriques.

   Ainsi la $\k$-vari\'et\'e $ W\times_{k}\k$ est isomorphe \`a l'\'eclat\'e d'une surface cubique lisse
 en 3 $\k$-points align\'es. 
 Ceci montre d\'ej\`a que $W$ est une 
 $k$-surface projective et lisse g\'eom\'etriquement rationnelle
dont le faisceau canonique $K$ satisfait $(K.K)=0$.
 
La section nulle $M$ de $W\to L$  correspond sur $\k$
\`a l'\'eclatement d'un $\k$-point de $\X$ non situ\'e sur une droite de $\X$,
elle satisfait $(M.M)=-1$. On peut donc la contracter.  
On obtient une surface $W'$ qui est l'\'eclat\'ee de la surface $\X$ en deux
$\k$-points non situ\'es sur les 27 droites, et 
dont le faisceau canonique satisfait $(K.K)=1$.
Pourvu que l'on ait pris la droite $D$ initiale dans un ouvert de Zariski non vide convenable
de la grassmannienne des  droites de $\P^3_{k}$, la surface $W'$
est g\'eom\'etriquement  l'\'eclat\'ee  de  la surface cubique
en un couple g\'en\'eral de points de $\X$, et donc est une
surface de del Pezzo de degr\'e 1.

Soit $m \in L(k)$ un point \`a fibre $Y_{m}$ lisse. La jacobienne de $Y_{m}$
est la fibre $W_{m}$ de $W \to L$ en $m$.

On a la suite exacte bien connue  faisant intervenir groupes de Picard et groupes de Brauer :
$$0 \to \Pic(Y_{m}) \to \Pic(\overline{Y}_{m})^G \to {\rm Br}(k) \to {\rm Br}(Y_{m}).$$
Comme $Y_{m}$ poss\`ede un point dans une extension de degr\'e 3 de $k$,
cette suite induit une suite exacte
$$0 \to \Pic(Y_{m}) \to \Pic(\overline{Y}_{m})^G \to {\rm Br}(k)[3]$$
o\`u $A[n]$ d\'esigne le sous-groupe de $n$-torsion d'un groupe ab\'elien $A$.
Sur les classes de degr\'e z\'ero, cette suite exacte induit une suite exacte
$$0 \to \Pic^0(Y_{m}) \to \Pic^0(\overline{Y}_{m})^G \to {\rm Br}(k)[3].$$
 Le groupe $\Pic^0(\overline{Y}_{m})^G$ est le groupe $W_{m}(k)$ des $k$-points   de
 la $k$-courbe elliptique $W_{m}$.
 Si la fl\`eche induite $W_{m}(k)  \to  {\rm Br}(k)[3]$ a un noyau non trivial,
 alors on a  $ \Pic^0(Y_{m})\neq 0$. Si $z$ est un \'el\'ement non nul dans  $ \Pic^0(Y_{m})$,
 alors la classe $Q+z$ est une classe de degr\'e~3, rationnellement \'equivalente 
 \`a un z\'ero-cycle effectif de degr\'e~3 par le th\'eor\`eme de Riemann-Roch sur
 $Y_{m}$, d\'efinissant un point ferm\'e de degr\'e~3 non rationnellement \'equivalent
 \`a $Q$ sur $Y_{m}$, et donc non align\'e.
 
La fl\`eche $W_{m}(k)  \to  {\rm Br}(k)[3]$ a un noyau non trivial 
si $W_{m}(k) \neq W_{m}(k)[3]$.

Si pour aucun $m \in L(k)$ \`a fibre lisse cette condition n'est  satisfaite,
les $k$-points de $W$ sont contenus dans la r\'eunion des  fibres singuli\`eres de $W \to \P^1_{k}$
et du ferm\'e de $W$ qui sur l'ouvert de lissit\'e correspond au sch\'ema
d\'efini par la 3-torsion. En particulier, les $k$-points ne sont pas denses
pour la topologie de Zariski sur  $W$, qui est une surface de del Pezzo de degr\'e 1.
\end{proof}

\begin{rmk}  Si le corps $k$ est fertile, par exemple si $k$ est un corps $p$-adique,
alors  pour toute surface de del Pezzo de degr\'e 1, l'ensemble $W(k)$, qui est non vide,
 est dense pour la topologie de Zariski. Dans ce cas on a donc des points ferm\'es
 de degr\'e 3 non align\'es sur toute $k$-surface cubique lisse.
 
 Ceci est en fait facile \`a voir directement. Soit $U \subset {\rm Gr}(1,\P^3_{k})$
 l'ouvert form\'e des droites qui rencontrent $X$ g\'eom\'etriquement en trois 
 point distincts. Soit $V \subset {\rm Sym}^3X$ l'ouvert  lisse, int\`egre, correspondant
 aux ensembles de 3 points g\'eom\'etriques distincts. 
 On a un $k$-plongement ferm\'e de $U$, de dimension~4, dans $V$, de dimension~6,
 identifiant $U(\k)$ avec les triplets de points g\'eom\'etriques distincts align\'es.
 L'image de $U(k)$ dans $V(k)$ d\'efinit des $k$-points, lisses, de $U$. Ainsi $V(k)$ est non vide,
et donc,  si $k$ est fertile, les points de $V (k)$ sont denses pour la topologie de Zariski sur $V$,
 et   il en existe hors de $U$.
  \end{rmk}
 
\begin{rmk} 
C'est une question ouverte   si pour toute surface de del Pezzo
de degr\'e 1 sur un corps $k$ de caract\'eristique z\'ero, les points rationnels
sont denses pour la topologie de Zariski.  On consultera
\cite{SvL} pour  des r\'esultats partiels. Plus g\'en\'eralement,
c'est aussi une question ouverte
si, pour une famille lisse non g\'eo\-m\'e\-tri\-quement isotriviale $W \to \P^1_{k}$
 de fibre g\'en\'erique une courbe elliptique, les points rationnels sont
 denses pour la topologie de Zariski. Ces probl\`emes sont ouverts 
 d\'ej\`a pour $k$ le corps $\Bbb Q$ des rationnels.
 \end{rmk}
 
 \begin{rmk}
 Soit $k=\Q$. Soient $p, q$  deux nombres premiers distincts, et distincts de $3$.
Soit $C \subset \P^2_{\Q}$ la cubique lisse sur $\Q$ d\'efinie par l'\'equation 
  $$x^3+pqy^3+q^2z^3=0.$$
  On a $C(\Q_{q})=\emptyset$, donc $C(\Q)=\emptyset$.
  La courbe jacobienne  $J$ de $C$ est donn\'ee par
  $$x^3+y^3+pz^3=0.$$
 Si $p$ est congru \`a 5 modulo 9,
  on sait \cite[Chap. 15, Thm. 3]{Mo} que l'on a $J(\Q)=0$.
  Comme on a ${\rm Pic}^0(C) \subset J(\Q)$, 
  ceci implique ${\rm Pic}^0(C)=0$.
  En utilisant  le th\'eor\`eme de Riemann-Roch sur la courbe $C$, on  en d\'eduit que
  tout point ferm\'e de degr\'e 3 sur $C$ est align\'e.  
  Ceci r\'epond \`a une question de C.~Shramov.
    \end{rmk}

 La d\'emonstration du th\'eor\`eme suivant est enti\`erement parall\`ele 
 \`a celle du th\'eor\`eme \ref{aligne} et est laiss\'ee au lecteur.
 
 \begin{thm}
Soit $X$ une surface del Pezzo de degr\'e 2 sur un corps $k$ de caract\'eristique nulle,
 sans point rationnel. Si tout point ferm\'e de degr\'e 2  sur $X$ 
est  image r\'eciproque d'un point rationnel de $\P^2_{k}$ via le morphisme
 anticanonique $X \to \P^2_{k}$, alors
il existe une surface de del Pezzo de degr\'e 1 sur $k$
 dont les points $k$-rationnels ne sont pas denses pour la topologie de Zariski,
 et donc qui en particulier n'est pas $k$-unirationnelle.
\end{thm}

\end{document}